\documentclass[11pt]{preprint}
\usepackage{difftrees} 
\usepackage[full]{textcomp}
\usepackage[osf]{newtxtext} 
\usepackage[cal=boondoxo]{mathalfa}
\usepackage{colortbl}

\usepackage{tikz-cd} 
\usetikzlibrary{cd} 
\usepackage[all,cmtip]{xy}
\usepackage{comment}

\usepackage{amssymb}
\usepackage{mathtools}
\usepackage{hyperref}
\usepackage{breakurl}
\usepackage{mhenvs}
\usepackage{mhequ} 
\usepackage{mhsymb}
\usepackage{booktabs}
\usepackage{tikz}
\usepackage{tcolorbox}
\usepackage{mathrsfs}
\usepackage[utf8]{inputenc}
\usepackage{longtable}
\usepackage{wrapfig}
\usepackage{rotating} 
\usepackage{subcaption}
\usepackage{mathrsfs}
\usepackage{epsfig}
\usepackage{microtype}
\usepackage{comment}
\usepackage{wasysym}
\usepackage{centernot}
\usepackage{enumitem}
\usepackage{bm}
\usepackage{stackrel}
\usepackage{graphicx}
\usepackage{axodraw}
\usepackage{xspace}
\usepackage{subcaption} 
\usepackage{epsfig} 
\usepackage{axodraw} 
\usepackage{xspace} 
\usepackage[toc,page]{appendix} 
\usepackage[all,cmtip]{xy} 
\usepackage{relsize} 
\usepackage{shuffle} 
\makeatletter
\newcommand{\globalcolor}[1]{%
  \color{#1}\global\let\default@color\current@color
}
\makeatother

\usetikzlibrary{calc}
\usetikzlibrary{decorations}
\usetikzlibrary{positioning}
\usetikzlibrary{shapes}
\usetikzlibrary{external}

\definecolor{blush}{rgb}{0.87, 0.36, 0.51}
	\definecolor{brightcerulean}{rgb}{0.11, 0.67, 0.84}
	\definecolor{greenryb}{rgb}{0.4, 0.69, 0.2}

\newif\ifdark
\darkfalse

\ifdark
\definecolor{darkred}{rgb}{0.9,0.2,0.2}
\definecolor{darkblue}{rgb}{0.7,0.3,1}
\definecolor{darkgreen}{rgb}{0.1,0.9,0.1}
\definecolor{franck}{rgb}{0,0.8,1}
\definecolor{pagebackground}{rgb}{.15,.21,.18}
\definecolor{pageforeground}{rgb}{.84,.84,.85}
\pagecolor{pagebackground}
\AtBeginDocument{\globalcolor{pageforeground}}
\definecolor{symbols}{rgb}{0,0.7,1}
\colorlet{connection}{red!80!black}
\colorlet{boxcolor}{blue!50}

\else

\definecolor{darkred}{rgb}{0.7,0.1,0.1}
\definecolor{darkblue}{rgb}{0.4,0.1,0.8}
\definecolor{darkgreen}{rgb}{0.1,0.7,0.1}
\definecolor{franck}{rgb}{0,0,1}
\definecolor{pagebackground}{rgb}{1,1,1}
\definecolor{pageforeground}{rgb}{0,0,0}
\colorlet{symbols}{blue!90!black}
\colorlet{connection}{red!30!black}
\colorlet{boxcolor}{blue!50!black}

\fi

\def\slash{\leavevmode\unskip\kern0.18em/\penalty\exhyphenpenalty\kern0.18em}
\def\dash{\leavevmode\unskip\kern0.18em--\penalty\exhyphenpenalty\kern0.18em}

\DeclareMathAlphabet{\mathbbm}{U}{bbm}{m}{n}

\DeclareFontFamily{U}{BOONDOX-calo}{\skewchar\font=45 }
\DeclareFontShape{U}{BOONDOX-calo}{m}{n}{
  <-> s*[1.05] BOONDOX-r-calo}{}
\DeclareFontShape{U}{BOONDOX-calo}{b}{n}{
  <-> s*[1.05] BOONDOX-b-calo}{}
\DeclareMathAlphabet{\mcb}{U}{BOONDOX-calo}{m}{n}
\SetMathAlphabet{\mcb}{bold}{U}{BOONDOX-calo}{b}{n}

\setlist{noitemsep,topsep=4pt,leftmargin=1.5em}

\DeclareMathAlphabet{\mathbbm}{U}{bbm}{m}{n}

\DeclareMathAlphabet{\mcb}{U}{BOONDOX-calo}{m}{n}
\SetMathAlphabet{\mcb}{bold}{U}{BOONDOX-calo}{b}{n}
\DeclareFontFamily{U}{mathx}{\hyphenchar\font45}
\DeclareFontShape{U}{mathx}{m}{n}{
      <5> <6> <7> <8> <9> <10>
      <10.95> <12> <14.4> <17.28> <20.74> <24.88>
      mathx10
      }{}
\DeclareSymbolFont{mathx}{U}{mathx}{m}{n}
\DeclareMathSymbol{\bigtimes}{1}{mathx}{"91}

\def\s{\mathfrak{s}}

\providecommand{\figures}{false}
{ \ifthenelse{\equal{\figures}{false}} {#1}{\[ {\rm Figure \ missing !} \]} }{}
\def\id{\mathrm{id}}

\def\CH{\mathcal{H}}
\def\CP{\mathcal{P}}

\def\CG{\mathcal{G}}
\def\CJ{\mathcal{J}}

\def\CC{\mathcal{C}}

\def\CB{\mathcal{B}}

\def\CT{\mathcal{T}}

\tikzstyle{tinydots}=[dash pattern=on \pgflinewidth off \pgflinewidth]
\tikzstyle{superdense}=[dash pattern=on 4pt off 1pt]

\newcommand{\BR}{\bf{BRP}}
\newcommand{\AN}{\bf{ARP}}

\newcommand{\beq}{\begin{equation}}
\newcommand{\eeq}{\end{equation}}

\usepackage{empheq}

\def\Labe{\mathfrak{e}}

\def\Labn{\mathfrak{n}}

\def\${|\!|\!|}

\newcounter{theorem}

\newtheorem{ex}[theorem]{Example}

\newenvironment{DIFnomarkup}{}{} 

\theorembodyfont{\rmfamily}

 \newcommand{\wh}{\widehat}

\newcommand{\rrightarrow}{{\to\hskip -4.9mm\raise 1pt\hbox{$\to$}}}

\newfont{\indic}{bbmss12}

\def\Nabla_#1{\nabla_{\!#1}}

\makeatletter
\pgfdeclareshape{crosscircle}
{
  \inheritsavedanchors[from=circle] 
  \inheritanchorborder[from=circle]
  \inheritanchor[from=circle]{north}
  \inheritanchor[from=circle]{north west}
  \inheritanchor[from=circle]{north east}
  \inheritanchor[from=circle]{center}
  \inheritanchor[from=circle]{west}
  \inheritanchor[from=circle]{east}
  \inheritanchor[from=circle]{mid}
  \inheritanchor[from=circle]{mid west}
  \inheritanchor[from=circle]{mid east}
  \inheritanchor[from=circle]{base}
  \inheritanchor[from=circle]{base west}
  \inheritanchor[from=circle]{base east}
  \inheritanchor[from=circle]{south}
  \inheritanchor[from=circle]{south west}
  \inheritanchor[from=circle]{south east}
  \inheritbackgroundpath[from=circle]
  \foregroundpath{
    \centerpoint%
    \pgf@xc=\pgf@x%
    \pgf@yc=\pgf@y%
    \pgfutil@tempdima=\radius%
    \pgfmathsetlength{\pgf@xb}{\pgfkeysvalueof{/pgf/outer xsep}}%
    \pgfmathsetlength{\pgf@yb}{\pgfkeysvalueof{/pgf/outer ysep}}%
    \ifdim\pgf@xb<\pgf@yb%
      \advance\pgfutil@tempdima by-\pgf@yb%
    \else%
      \advance\pgfutil@tempdima by-\pgf@xb%
    \fi%
    \pgfpathmoveto{\pgfpointadd{\pgfqpoint{\pgf@xc}{\pgf@yc}}{\pgfqpoint{-0.707107\pgfutil@tempdima}{-0.707107\pgfutil@tempdima}}}
    \pgfpathlineto{\pgfpointadd{\pgfqpoint{\pgf@xc}{\pgf@yc}}{\pgfqpoint{0.707107\pgfutil@tempdima}{0.707107\pgfutil@tempdima}}}
    \pgfpathmoveto{\pgfpointadd{\pgfqpoint{\pgf@xc}{\pgf@yc}}{\pgfqpoint{-0.707107\pgfutil@tempdima}{0.707107\pgfutil@tempdima}}}
    \pgfpathlineto{\pgfpointadd{\pgfqpoint{\pgf@xc}{\pgf@yc}}{\pgfqpoint{0.707107\pgfutil@tempdima}{-0.707107\pgfutil@tempdima}}}
  }
}
\makeatother

\def\symbol#1{\textcolor{symbols}{#1}}

\def\decorate#1#2{
        \ifnum#2>0
    		\foreach \count in {1,...,#2}{
	       	let
				\p1 = (sourcenode.center),
                \p2 = (sourcenode.east),
				\n1 = {\x2-\x1},
				\n2 = {1mm},
				\n3 = {(1.3+0.6*(\count-1))*\n1},
				\n4 = {0.7*\n1}
			in 
        		node[rectangle,fill=symbols,rotate=30,inner sep=0pt,minimum width=0.2*\n2,minimum height=\n2] at ($(sourcenode.center) + (\n3,\n4)$) {}
				}
		\fi
        \ifnum#1>0
    		\foreach \count in {1,...,#1}{
	       	let
				\p1 = (sourcenode.center),
                \p2 = (sourcenode.east),
				\n1 = {\x2-\x1},
				\n2 = {1mm},
				\n3 = {(1.3+0.6*(\count-1))*\n1},
				\n4 = {0.7*\n1}
			in 
        		node[rectangle,fill=symbols,rotate=-30,inner sep=0pt,minimum width=0.2*\n2,minimum height=\n2] at ($(sourcenode.center) + (-\n3,\n4)$) {}
				}
		\fi
}

\tikzset{
    dectriangle/.style 2 args={
        triangle,
        alias=sourcenode,
        append after command={\decorate{#1}{#2}}
    },
    dectriangle/.default={0}{0},
}

\tikzset{
	cross/.style={path picture={ 
  		\draw[symbols]
			(path picture bounding box.south east) -- (path picture bounding box.north west) (path picture bounding box.south west) -- (path picture bounding box.north east);
		}},
root/.style={circle,fill=green!50!black,inner sep=0pt, minimum size=1.2mm},
        dot/.style={circle,fill=pageforeground,inner sep=0pt, minimum size=1mm},
        dotred/.style={circle,fill=pageforeground!50!pagebackground,inner sep=0pt, minimum size=2mm},
        var/.style={circle,fill=pageforeground!10!pagebackground,draw=pageforeground,inner sep=0pt, minimum size=3mm},
        kernel/.style={semithick,shorten >=2pt,shorten <=2pt},
        kernels/.style={snake=zigzag,shorten >=2pt,shorten <=2pt,segment amplitude=1pt,segment length=4pt,line before snake=2pt,line after snake=5pt,},
        rho/.style={densely dashed,semithick,shorten >=2pt,shorten <=2pt},
           testfcn/.style={dotted,semithick,shorten >=2pt,shorten <=2pt},
        renorm/.style={shape=circle,fill=pagebackground,inner sep=1pt},
        labl/.style={shape=rectangle,fill=pagebackground,inner sep=1pt},
        xic/.style={very thin,circle,draw=symbols,fill=symbols,inner sep=0pt,minimum size=1.2mm},
        g/.style={very thin,rectangle,draw=symbols,fill=symbols!10!pagebackground,inner sep=0pt,minimum width=2.5mm,minimum height=1.2mm},
        xi/.style={very thin,circle,draw=symbols,fill=symbols!10!pagebackground,inner sep=0pt,minimum size=1.2mm},
	xies/.style={very thin,rectangle,fill=green!50!black!25,draw=symbols,inner sep=0pt,minimum size=1.1mm},
	xiesf/.style={very thin,rectangle,fill=green!50!black,draw=symbols,inner sep=0pt,minimum size=1.1mm},
        xix/.style={very thin,crosscircle,fill=symbols!10!pagebackground,draw=symbols,inner sep=0pt,minimum size=1.2mm},
        X/.style={very thin,cross,rectangle,fill=pagebackground,draw=symbols,inner sep=0pt,minimum size=1.2mm},
	xib/.style={thin,circle,fill=symbols!10!pagebackground,draw=symbols,inner sep=0pt,minimum size=1.6mm},
	xie/.style={thin,circle,fill=green!50!black,draw=symbols,inner sep=0pt,minimum size=1.6mm},
	xid/.style={thin,circle,fill=symbols,draw=symbols,inner sep=0pt,minimum size=1.6mm},
	xibx/.style={thin,crosscircle,fill=symbols!10!pagebackground,draw=symbols,inner sep=0pt,minimum size=1.6mm},
	kernels2/.style={very thick,draw=connection,segment length=12pt},
	keps/.style={thin,draw=symbols,->},
	kepspr/.style={thick,draw=connection,->},
	krho/.style={thin,draw=symbols,superdense,->},
	krhopr/.style={thick,draw=connection,superdense},
	triangle/.style = { regular polygon, regular polygon sides=3},
	not/.style={thin,circle,draw=connection,fill=connection,inner sep=0pt,minimum size=0.5mm},
	diff/.style = {very thin,draw=symbols,triangle,fill=red!50!black,inner sep=0pt,minimum size=1.6mm},
	diff1/.style = {very thin,dectriangle={1}{0},fill=red!50!black,draw=symbols,inner sep=0pt,minimum size=1.6mm},
	diff2/.style = {very thin,dectriangle={1}{1},fill=red!50!black,draw=symbols,inner sep=0pt,minimum size=1.6mm},
		diffmini/.style = {very thin,rectangle,fill=black,draw=black,inner sep=0pt,minimum size=0.75mm},
	 kernelsmod/.style={very thick,draw=connection,segment length=12pt},
	 rec/.style = {very thin,rectangle,fill=black,draw=black,inner sep=0pt,minimum size=2mm},
	cerc/.style={very thin,circle,draw=black,fill=symbols,inner sep=0pt,minimum size=2mm},
	stars/.style={very thin,star,star points=6,star point ratio=0.5, draw=black,fill=red,inner sep=0pt,minimum size=0.7mm},
	>=stealth,
        }
        \tikzset{
root/.style={circle,fill=black!50,inner sep=0pt, minimum size=3mm},
        circ/.style={circle,fill=white,draw=black,very thin,inner sep=.5pt, minimum size=1.2mm},
        round1/.style={fill=white,outer sep = 0,inner sep=2pt,rounded corners=1mm,draw,text=black,thin,minimum size=1.2mm},
          circ1/.style={circle,fill=red!10,draw=red,very thin,inner sep=.5pt, minimum size=1.2mm},
        rect/.style={fill=white,outer sep = 0,inner sep=2pt,rectangle,draw,text=black,thin,minimum size=1.2mm},
        rect1/.style={fill=white,outer sep = 0,inner sep=2pt,rectangle,draw,text=black,thin,minimum size=1.2mm},
        round2/.style={fill=red!10,outer sep = 0,inner sep=2pt,rounded corners=1mm,draw,text=black,thin,minimum size=1.2mm},
       round3/.style={fill=blue!10,outer sep = 0,inner sep=2pt,rounded corners=1mm,draw,text=black,thin,minimum size=1.2mm}, 
        rect2/.style={fill=black!10,outer sep = 0,inner sep=2pt,rectangle,draw,text=black,thin,minimum size=1.2mm},
        dot/.style={circle,fill=black,inner sep=0pt, minimum size=1.2mm},
        dotred/.style={circle,fill=black!50,inner sep=0pt, minimum size=2mm},
        var/.style={circle,fill=black!10,draw=black,inner sep=0pt, minimum size=3mm},
        kernel/.style={semithick,shorten >=2pt,shorten <=2pt},
         diag/.style={thin,shorten >=4pt,shorten <=4pt},
        kernel1/.style={thick},
        kernels/.style={snake=zigzag,shorten >=2pt,shorten <=2pt,segment amplitude=1pt,segment length=4pt,line before snake=2pt,line after snake=5pt,},
		kernels1/.style={snake=zigzag,segment amplitude=0.5pt,segment length=2pt},
		rho1/.style={densely dotted,semithick},
        rho/.style={densely dashed,semithick,shorten >=2pt,shorten <=2pt},
           testfcn/.style={dotted,semithick,shorten >=2pt,shorten <=2pt},
           visible/.style={draw, circle, fill, inner sep=0.25ex},
        renorm/.style={shape=circle,fill=white,inner sep=1pt},
        labl/.style={shape=rectangle,fill=white,inner sep=1pt},
        xic/.style={very thin,circle,fill=symbols,draw=black,inner sep=0pt,minimum size=1.2mm},
        xi/.style={very thin,circle,fill=blue!10,draw=black,inner sep=0pt,minimum size=1.2mm},
	xib/.style={very thin,circle,fill=blue!10,draw=black,inner sep=0pt,minimum size=1.6mm},
	xie/.style={very thin,circle,fill=green!50!black,draw=black,inner sep=0pt,minimum size=1mm},
	xid/.style={very thin,circle,fill=symbols,draw=black,inner sep=0pt,minimum size=1.6mm},
	edgetype/.style={very thin,circle,draw=black,inner sep=0pt,minimum size=5mm},
	nodetype/.style={very thick,circle,draw=black,inner sep=0pt,minimum size=5mm},
	kernels2/.style={very thick,draw=connection,segment length=12pt},
clean/.style={thin,circle,fill=black,inner sep=0pt,minimum size=1mm},	not/.style={thin,circle,fill=symbols,draw=connection,fill=connection,inner sep=0pt,minimum size=0.8mm},
	>=stealth,
        }

\makeatletter
\def\DeclareSymbol#1#2#3{%
	\expandafter\gdef\csname MH@symb@#1\endcsname{\tikzsetnextfilename{symbol#1}%
	\tikz[baseline=#2,scale=0.15,draw=symbols,line join=round]{#3}}%
	\expandafter\gdef\csname MH@symb@#1s\endcsname{\scalebox{0.75}{\tikzsetnextfilename{symbol#1}%
	\tikz[baseline=#2,scale=0.15,draw=symbols,line join=round]{#3}}}%
	\expandafter\gdef\csname MH@symb@#1ss\endcsname{\scalebox{0.65}{\tikzsetnextfilename{symbol#1}%
	\tikz[baseline=#2,scale=0.15,draw=symbols,line join=round]{#3}}}%
	}
\def\<#1>{\ifthenelse{\boolean{mmode}}{\mathchoice{\csname MH@symb@#1\endcsname}{\csname MH@symb@#1\endcsname}{\csname MH@symb@#1s\endcsname}{\csname MH@symb@#1ss\endcsname}}{\csname MH@symb@#1\endcsname}}
\makeatother

\DeclareSymbol{Xi22}{0.5}{\draw (0,0) node[xi] {} -- (-1,1) node[not] {} -- (0,2) node[xi] {};} 
\DeclareSymbol{Xi2}{-2}{\draw (-1,-0.25) node[xi] {} -- (0,1) node[xi] {};} 
\DeclareSymbol{Xi2b}{-2}{\draw (-1,-0.25) node[xic] {} -- (0,1) node[xic] {};} 
\DeclareSymbol{Xi2g}{-2}{\draw (-1,-0.25) node[xies] {} -- (0,1) node[xi] {};} 
\DeclareSymbol{Xi2g2}{-2}{\draw (-1,-0.25) node[xi] {} -- (0,1) node[xies] {};} 
\DeclareSymbol{cXi2}{-2}{\draw (0,-0.25) node[xi] {} -- (-1,1) node[xic] {};}
\DeclareSymbol{Xi3}{0}{\draw (0,0) node[xi] {} -- (-1,1) node[xi] {} -- (0,2) node[xi] {};}
\DeclareSymbol{XiIIXi}{0}{\draw (0,0) node[xi] {} -- (-1,1); \draw[kernels2] (-1,1) node[not] {} -- (0,2) node[xi] {};}

\DeclareSymbol{Xi4}{2}{\draw (0,0) node[xi] {} -- (-1,1) node[xi] {} -- (0,2) node[xi] {} -- (-1,3) node[xi] {};}
\DeclareSymbol{Xi4_1}{2}{\draw (0,0) node[xic] {} -- (-1,1) node[xic] {} -- (0,2) node[xi] {} -- (-1,3) node[xi] {};}
\DeclareSymbol{Xi4_2}{2}{\draw (0,0) node[xic] {} -- (-1,1) node[xi] {} -- (0,2) node[xi] {} -- (-1,3) node[xic] {};}
\DeclareSymbol{Xi2X}{-2}{\draw (0,-0.25) node[xi] {} -- (-1,1) node[xix] {};}
\DeclareSymbol{XXi2}{-2}{\draw (0,-0.25) node[xix] {} -- (-1,1) node[xi] {};}
\DeclareSymbol{IIXi}{0}{\draw (0,-0.25) node[not] {} -- (-1,1) node[xi] {} -- (0,2) node[xi] {};}
\DeclareSymbol{IXi^2}{-1}{\draw (-1,1) node[xi] {} -- (0,0) node[not] {} -- (1,1) node[xi] {};}
\DeclareSymbol{IIXi^2}{-4}{\draw (0,-1.5) node[not] {} -- (0,0);
\draw[kernels2] (-1,1) node[xi] {} -- (0,0) node[not] {} -- (1,1) node[xi] {};}
\DeclareSymbol{XiX}{-2.8}{\node[xibx] {};}
\DeclareSymbol{tauX}{-2.8}{ \node[X] {};}
\DeclareSymbol{Xi}{-2.8}{\node[xib] {};}

\DeclareSymbol{IXiX}{-1}{\draw (0,-0.25) node[not] {} -- (-1,1) node[xix] {};}
\DeclareSymbol{IXi3}{2}{\draw (0,-0.25) node[not] {} -- (-1,1) node[xi] {} -- (0,2) node[xi] {} -- (-1,3) node[xi] {};}
\DeclareSymbol{IXi}{-2}{\draw (0,-0.25) node[not] {} -- (-1,1) node[xi] {};}
\DeclareSymbol{XiI}{-2}{\draw (0,-0.25) node[xi] {} -- (-1,1) node[not] {};}

\DeclareSymbol{Xi4b}{0}{\draw(0,1.5) node[xi] {} -- (0,0); \draw (-1,1) node[xi] {} -- (0,0) node[xi] {} -- (1,1) node[xi] {};}
\DeclareSymbol{Xi4b'}{0}{\draw(0,1.5) node[xi] {} -- (0,-0.2); \draw (-1,1) node[xi] {} -- (0,-0.2) node[not] {} -- (1,1) node[xi] {};}
\DeclareSymbol{Xi4c}{0}{\draw (0,1) -- (0.8,2.2) node[xi] {};\draw (0,-0.25) node[xi] {} -- (0,1) node[xi] {} -- (-0.8,2.2) node[xi] {};}
\DeclareSymbol{Xi4d}{-4.5}{\draw (0,-1.5) node[not] {} -- (0,0); \draw (-1,1) node[xi] {} -- (0,0) node[xi] {} -- (1,1) node[xi] {};}
\DeclareSymbol{Xi4e}{0}{\draw (0,2) node[xi] {} -- (-1,1) node[xi] {} -- (0,0) node[xi] {} -- (1,1) node[xi] {};}
\DeclareSymbol{Xi4e'}{0}{\draw (0,2) node[xi] {} -- (-1,1) node[xi] {} -- (0,-0.2) node[not] {} -- (1,1) node[xi] {};}

\DeclareSymbol{Xitwo}
{0}{\draw[kernels2] (0,0) node[not] {} -- (-1,1) node[not] {}
-- (-2,2) node[not]{} -- (-3,3) node[xi]  {};
\draw[kernels2] (0,0) -- (1,1) node[xi] {};
\draw[kernels2] (-1,1) -- (0,2) node[xi] {};
\draw[kernels2] (-2,2) -- (-1,3) node[xi] {};}

\DeclareSymbol{IXitwo}
{0}{\draw (-.7,1.2) node[xi] {} -- (0,-0.2) -- (.7,1.2) node[xi] {};}
\DeclareSymbol{I1Xitwo}
{0}{\draw[kernels2] (0,0) node[not] {} -- (-1,1) node[xi] {};
\draw[kernels2] (0,0) -- (1,1) node[xi] {};}

\DeclareSymbol{I1Xitwobis}
{0}{\draw[kernels2] (0,0) node[not] {} -- (-1,1) node[xies] {};
\draw[kernels2] (0,0) -- (1,1) node[xies] {};}

\DeclareSymbol{I1Xitwog}
{0}{\draw[kernels2] (0,0) node[not] {} -- (-1,1) node[xies] {};
\draw[kernels2] (0,0) -- (1,1) node[xi] {};}

\DeclareSymbol{cI1Xitwo}
{0}{\draw[kernels2] (0,0) node[not] {} -- (-1,1) node[xic] {};
\draw[kernels2] (0,0) -- (1,1) node[xi] {};}

\DeclareSymbol{I1IXi3}{0}{\draw (0,0) node[xi] {} -- (-1,1) ; 
\draw[kernels2] (-1,1) node[not] {} -- (0,2) node[xi] {};
\draw[kernels2] (-1,1) node[not] {} -- (-2,2) node[xi] {};}

\DeclareSymbol{I1Xi3c}{-1}{\draw[kernels2](0,1.5) node[xi] {} -- (0,0) node[not] {}; \draw (-1,1) node[xi] {} -- (0,0) ; \draw[kernels2] (0,0) -- (1,1) node[xi] {};}

\DeclareSymbol{I1Xi3cbis}{-1}{\draw[kernels2](0,1.5) node[xies] {} -- (0,0) node[not] {}; \draw (-1,1) node[xies] {} -- (0,0) ; \draw[kernels2] (0,0) -- (1,1) node[xies] {};}

\DeclareSymbol{I1IXi3b}{0}{\draw[kernels2] (0,0) node[not] {} -- (-1,1) ; \draw[kernels2] (0,0)   -- (1,1) node[xi] {} ;
\draw (-1,1) node[xi] {} -- (0,2) node[xi] {};
}

\DeclareSymbol{I1IXi3c}{0}{\draw[kernels2] (0,0) node[not] {} -- (-1,1) ; \draw[kernels2] (0,0)   -- (1,1) node[xi] {} ;
\draw[kernels2] (-1,1) node[not] {} -- (0,2) node[xi] {};
\draw[kernels2] (-1,1) node[not] {} -- (-2,2) node[xi] {};}

\DeclareSymbol{I1IXi3cbis}{0}{\draw[kernels2] (0,0) node[not] {} -- (-1,1) ; \draw[kernels2] (0,0)   -- (1,1) node[xies] {} ;
\draw[kernels2] (-1,1) node[not] {} -- (0,2) node[xies] {};
\draw[kernels2] (-1,1) node[not] {} -- (-2,2) node[xies] {};}

\DeclareSymbol{I1Xi}{0}{\draw[kernels2] (0,0) node[not] {} -- (-1,1)  node[xi] {} ;}

\DeclareSymbol{I1Xi4a}{2}{\draw[kernels2] (0,0) node[not] {} -- (-1,1) ; \draw[kernels2] (0,0) node[not] {} -- (1,1) node[xi] {} ;
\draw (-1,1) node[xi] {} -- (0,2) node[xi] {} -- (-1,3) node[xi] {};}

\DeclareSymbol{cI1Xi4a}{2}{\draw[kernels2] (0,0) node[not] {} -- (-1,1) ; \draw[kernels2] (0,0) node[not] {} -- (1,1) node[xic] {} ;
\draw (-1,1) node[xic] {} -- (0,2) node[xi] {} -- (-1,3) node[xi] {};}

\DeclareSymbol{I1Xi4b}{2}{\draw (0,0) node[xi] {} -- (-1,1) node[xi] {} -- (0,2) ; \draw[kernels2] (0,2) node[not] {} -- (-1,3) node[xi] {};\draw[kernels2] (0,2)  -- (1,3) node[xi] {};
}

\DeclareSymbol{cI1Xi4b}{2}{\draw (0,0) node[xic] {} -- (-1,1) node[xic] {} -- (0,2) ; \draw[kernels2] (0,2) node[not] {} -- (-1,3) node[xi] {};\draw[kernels2] (0,2)  -- (1,3) node[xi] {};
}

\DeclareSymbol{I1Xi4c}{2}{\draw (0,0) node[xi] {} -- (-1,1) node[not] {}; \draw[kernels2] (-1,1) -- (0,2) ; 
\draw[kernels2] (-1,1) -- (-2,2) node[xi] {} ;
\draw (0,2) node[xi] {} -- (-1,3) node[xi] {};}

\DeclareSymbol{cI1Xi4c}{2}{\draw (0,0) node[xic] {} -- (-1,1) node[not] {}; \draw[kernels2] (-1,1) -- (0,2) ; 
\draw[kernels2] (-1,1) -- (-2,2) node[xic] {} ;
\draw (0,2) node[xi] {} -- (-1,3) node[xi] {};}

\DeclareSymbol{I1Xi4ab}{2}{\draw[kernels2] (0,0) node[not] {} -- (-1,1) ; \draw[kernels2] (0,0) node[not] {} -- (1,1) node[xi] {};\draw (-1,1) node[xi] {} -- (0,2) ; \draw[kernels2] (0,2) node[not] {} -- (-1,3) node[xi] {};\draw[kernels2] (0,2)  -- (1,3) node[xi] {}; }

\DeclareSymbol{cI1Xi4ab}{2}{\draw[kernels2] (0,0) node[not] {} -- (-1,1) ; \draw[kernels2] (0,0) node[not] {} -- (1,1) node[xic] {};\draw (-1,1) node[xic] {} -- (0,2) ; \draw[kernels2] (0,2) node[not] {} -- (-1,3) node[xi] {};\draw[kernels2] (0,2)  -- (1,3) node[xi] {}; }

\DeclareSymbol{I1Xi4bc}{2}{\draw (0,0) node[xi] {} -- (-1,1) node[not] {}; \draw[kernels2] (-1,1) -- (0,2) ; 
\draw[kernels2] (-1,1) -- (-2,2) node[xi] {} ; \draw[kernels2] (0,2) node[not] {} -- (-1,3) node[xi] {};\draw[kernels2] (0,2)  -- (1,3) node[xi] {};
}

\DeclareSymbol{cI1Xi4bc}{2}{\draw (0,0) node[xic] {} -- (-1,1) node[not] {}; \draw[kernels2] (-1,1) -- (0,2) ; 
\draw[kernels2] (-1,1) -- (-2,2) node[xic] {} ; \draw[kernels2] (0,2) node[not] {} -- (-1,3) node[xi] {};\draw[kernels2] (0,2)  -- (1,3) node[xi] {};
}

\DeclareSymbol{I1Xi4abcc1}{2}{\draw[kernels2] (0,0) node[not] {} -- (-1,1) node[not] {}
-- (-2,2) node[not]{} -- (-3,3) node[xic]  {};
\draw[kernels2] (0,0) -- (1,1) node[xic] {};
\draw[kernels2] (-1,1) -- (0,2) node[xi] {};
\draw[kernels2] (-2,2) -- (-1,3) node[xi] {};
}

\DeclareSymbol{I1Xi4abcc1b}{2}{\draw[kernels2] (0,0) node[not] {} -- (-1,1) node[not] {}
-- (-2,2) node[not]{} -- (-3,3) node[xi]  {};
\draw[kernels2] (0,0) -- (1,1) node[xic] {};
\draw[kernels2] (-1,1) -- (0,2) node[xic] {};
\draw[kernels2] (-2,2) -- (-1,3) node[xi] {};
}

\DeclareSymbol{I1Xi4abcc2}{2}{\draw[kernels2] (0,0) node[not] {} -- (-1,1) node[not] {}
-- (-2,2) node[not]{} -- (-3,3) node[xic]  {};
\draw[kernels2] (0,0) -- (1,1) node[xi] {};
\draw[kernels2] (-1,1) -- (0,2) node[xi] {};
\draw[kernels2] (-2,2) -- (-1,3) node[xic] {};
}

\DeclareSymbol{I1Xi4ac}{2}{\draw[kernels2] (0,0) node[not] {} -- (-1,1) ; \draw[kernels2] (0,0) node[not] {} -- (1,1) node[xi] {}; 
\draw[kernels2] (-1,1) node[not] {} -- (0,2) ; 
\draw[kernels2] (-1,1) -- (-2,2) node[xi] {} ;
\draw (0,2) node[xi] {} -- (-1,3) node[xi] {};}

\DeclareSymbol{cI1Xi4ac}{2}{\draw[kernels2] (0,0) node[not] {} -- (-1,1) ; \draw[kernels2] (0,0) node[not] {} -- (1,1) node[xic] {}; 
\draw[kernels2] (-1,1) node[not] {} -- (0,2) ; 
\draw[kernels2] (-1,1) -- (-2,2) node[xic] {} ;
\draw (0,2) node[xi] {} -- (-1,3) node[xi] {};}

\DeclareSymbol{I1Xi4acc1}{2}{\draw[kernels2] (0,0) node[not] {} -- (-1,1) ; \draw[kernels2] (0,0) node[not] {} -- (1,1) node[xic] {}; 
\draw[kernels2] (-1,1) node[not] {} -- (0,2) ; 
\draw[kernels2] (-1,1) -- (-2,2) node[xi] {} ;
\draw (0,2) node[xic] {} -- (-1,3) node[xi] {};}

\DeclareSymbol{I1Xi4acc2}{2}{\draw[kernels2] (0,0) node[not] {} -- (-1,1) ; \draw[kernels2] (0,0) node[not] {} -- (1,1) node[xic] {}; 
\draw[kernels2] (-1,1) node[not] {} -- (0,2) ; 
\draw[kernels2] (-1,1) -- (-2,2) node[xi] {} ;
\draw (0,2) node[xi] {} -- (-1,3) node[xic] {};}

\DeclareSymbol{2I1Xi4}{2}{\draw[kernels2] (0,0) node[not] {} -- (-1,1) node[not] {};
\draw[kernels2] (0,0) -- (1,1) node[not] {};
\draw[kernels2] (-1,1) -- (-1.5,2.5) node[xi] {};
\draw[kernels2] (-1,1) -- (-0.5,2.5) node[xi] {};
\draw[kernels2] (1,1) -- (0.5,2.5) node[xi] {};
\draw[kernels2] (1,1) -- (1.5,2.5) node[xi] {};
}

\DeclareSymbol{2I1Xi4dis}{2}{\draw[kernels2] (0,0) node[not] {} -- (-1,1) node[not] {};
\draw[kernels2] (0,0) -- (1,1) node[not] {};
\draw[kernels2] (-1,1) -- (-1.5,2.5) node[xies] {};
\draw[kernels2] (-1,1) -- (-0.5,2.5) node[xies] {};
\draw[kernels2] (1,1) -- (0.5,2.5) node[xies] {};
\draw[kernels2] (1,1) -- (1.5,2.5) node[xies] {};
}

\DeclareSymbol{2I1Xi4c1}{2}{\draw[kernels2] (0,0) node[not] {} -- (-1,1) node[not] {};
\draw[kernels2] (0,0) -- (1,1) node[not] {};
\draw[kernels2] (-1,1) -- (-1.5,2.5) node[xic] {};
\draw[kernels2] (-1,1) -- (-0.5,2.5) node[xi] {};
\draw[kernels2] (1,1) -- (0.5,2.5) node[xic] {};
\draw[kernels2] (1,1) -- (1.5,2.5) node[xi] {};
}

\DeclareSymbol{2I1Xi4c2}{2}{\draw[kernels2] (0,0) node[not] {} -- (-1,1) node[not] {};
\draw[kernels2] (0,0) -- (1,1) node[not] {};
\draw[kernels2] (-1,1) -- (-1.5,2.5) node[xic] {};
\draw[kernels2] (-1,1) -- (-0.5,2.5) node[xic] {};
\draw[kernels2] (1,1) -- (0.5,2.5) node[xi] {};
\draw[kernels2] (1,1) -- (1.5,2.5) node[xi] {};
}

\DeclareSymbol{2I1Xi4b}{2}{\draw[kernels2] (0,0) node[not] {} -- (-1,1) ;
\draw[kernels2] (0,0) -- (1,1);
\draw (-1,1) node[xi] {} -- (-1,2.5) node[xi] {};
\draw (1,1)  node[xi] {} -- (1,2.5) node[xi] {};
}

\DeclareSymbol{2I1Xi4bb}{2}{\draw[kernels2] (0,0) node[not] {} -- (-1,1) ;
\draw[kernels2] (0,0) -- (1,1);
\draw (-1,1) node[xi] {} -- (-1,2.5) node[xiesf] {};
\draw (1,1)  node[xi] {} -- (1,2.5) node[xic] {};
}

\DeclareSymbol{2I1Xi4c}{2}{\draw[kernels2] (0,0) node[not] {} -- (-1,1);
\draw[kernels2] (0,0) -- (1,1) node[not] {};
\draw (-1,1)  node[xi] {} -- (-1,2.5) node[xi] {};
\draw[kernels2] (1,1) -- (0.4,2.5) node[xi] {};
\draw[kernels2] (1,1) -- (1.6,2.5) node[xi] {};
}

\DeclareSymbol{2I1Xi4cc1}{2}{\draw[kernels2] (0,0) node[not] {} -- (-1,1);
\draw[kernels2] (0,0) -- (1,1) node[not] {};
\draw (-1,1)  node[xic] {} -- (-1,2.5) node[xi] {};
\draw[kernels2] (1,1) -- (0.4,2.5) node[xic] {};
\draw[kernels2] (1,1) -- (1.6,2.5) node[xi] {};
}

\DeclareSymbol{2I1Xi4cc2}{2}{\draw[kernels2] (0,0) node[not] {} -- (-1,1);
\draw[kernels2] (0,0) -- (1,1) node[not] {};
\draw (-1,1)  node[xic] {} -- (-1,2.5) node[xic] {};
\draw[kernels2] (1,1) -- (0.4,2.5) node[xi] {};
\draw[kernels2] (1,1) -- (1.6,2.5) node[xi] {};
}

\DeclareSymbol{Xi4ba}{0}{\draw(-0.5,1.5) node[xi] {} -- (0,0); \draw (-1.5,1) node[xi] {} -- (0,0) node[not] {}; \draw[kernels2] (0,0) -- (1.5,1) node[xi] {};
\draw[kernels2] (0,0) -- (0.5,1.5) node[xi] {} ;}

\DeclareSymbol{Xi4badis}{0}{\draw(-0.5,1.5) node[xies] {} -- (0,0); \draw (-1.5,1) node[xies] {} -- (0,0) node[not] {}; \draw[kernels2] (0,0) -- (1.5,1) node[xies] {};
\draw[kernels2] (0,0) -- (0.5,1.5) node[xies] {} ;}

\DeclareSymbol{Xi4ba1}{0}{\draw(-0.5,1.5) node[xi] {} -- (0,0); \draw (-1.5,1) node[xi] {} -- (0,0) node[not] {}; \draw[kernels2] (0,0) -- (1.5,1) node[xic] {};
\draw[kernels2] (0,0) -- (0.5,1.5) node[xic] {} ;}

\DeclareSymbol{Xi4ba1b}{0}{\draw(-0.5,1.5) node[xic] {} -- (0,0); \draw (-1.5,1) node[xic] {} -- (0,0) node[not] {}; \draw[kernels2] (0,0) -- (1.5,1) node[xi] {};
\draw[kernels2] (0,0) -- (0.5,1.5) node[xi] {} ;}

\DeclareSymbol{Xi4ba1bdiff}{0}{\draw(-0.5,1.5) node[xic] {} -- (0,0); \draw (-1.5,1) node[xic] {} -- (0,0) node[not] {}; \draw (0,0) -- (1.5,1) node[xi] {};
\draw (0,0) -- (0.5,1.5) node[xi] {};
\draw(0,0) node[diff] {};}

\DeclareSymbol{Xi4ba1bb}{0}{\draw(-0.5,1.5) node[xic] {} -- (0,0); \draw (-1.5,1) node[xiesf] {} -- (0,0) node[not] {}; \draw[kernels2] (0,0) -- (1.5,1) node[xi] {};
\draw[kernels2] (0,0) -- (0.5,1.5) node[xi] {} ;}

\DeclareSymbol{Xi4ba2}{0}{\draw(-0.5,1.5) node[xi] {} -- (0,0); \draw (-1.5,1) node[xic] {} -- (0,0) node[not] {}; \draw[kernels2] (0,0) -- (1.5,1) node[xi] {};
\draw[kernels2] (0,0) -- (0.5,1.5) node[xic] {} ;}

\DeclareSymbol{Xi4ba2b}{0}{\draw(-0.5,1.5) node[xi] {} -- (0,0); \draw (-1.5,1) node[xic] {} -- (0,0) node[not] {}; \draw[kernels2] (0,0) -- (1.5,1) node[xi] {};
\draw[kernels2] (0,0) -- (0.5,1.5) node[xiesf] {} ;}

\DeclareSymbol{Xi4ca}{0}{\draw (0,1) -- (-1,2.2) node[xi] {};\draw (0,-0.25) node[xi] {} -- (0,1) ; \draw[kernels2] (0,1) node[not] {} -- (1,2.2) node[xi] {};
\draw[kernels2] (0,1) {} -- (0,2.7) node[xi] {};
}

\DeclareSymbol{Xi4cb}{0}{\draw (-1,1) -- (-2,2) node[xi] {};\draw[kernels2] (0,0)  -- (-1,1) node[xi] {} ; \draw[kernels2] (0,0) node[not] {} -- (1,1) node[xi] {} ; 
\draw (-1,1) node[xi] {} -- (0,2) node[xi] {};}

\DeclareSymbol{Xi4cbb}{0}{\draw (-1,1) -- (-2,2) node[xiesf] {};\draw[kernels2] (0,0)  -- (-1,1) node[xi] {} ; \draw[kernels2] (0,0) node[not] {} -- (1,1) node[xi] {} ; 
\draw (-1,1) node[xi] {} -- (0,2) node[xic] {};}

\DeclareSymbol{Xi4cbc1}{0}{\draw (-1,1) -- (-2,2) node[xic] {};\draw[kernels2] (0,0)  -- (-1,1) node[xic] {} ; \draw[kernels2] (0,0) node[not] {} -- (1,1) node[xi] {} ; 
\draw (-1,1) node[xic] {} -- (0,2) node[xi] {};}

\DeclareSymbol{Xi4cbc2}{0}{\draw (-1,1) -- (-2,2) node[xi] {};\draw[kernels2] (0,0)  -- (-1,1) node[xi] {} ; \draw[kernels2] (0,0) node[not] {} -- (1,1) node[xic] {} ; 
\draw (-1,1) node[xic] {} -- (0,2) node[xi] {};}

\DeclareSymbol{Xi4cab}{0}{\draw (-1,1) -- (-2,2) node[xi] {};\draw[kernels2] (0,0)  -- (-1,1); \draw[kernels2] (0,0) node[not] {} -- (1,1) node[xi] {} ; 
\draw[kernels2] (-1,1)  {} -- (0,2) node[xi] {};
\draw[kernels2] (-1,1) node[not] {} -- (-1,2.5) node[xi] {};
}

\DeclareSymbol{Xi4cabdis}{0}{\draw (-1,1) -- (-2,2) node[xies] {};\draw[kernels2] (0,0)  -- (-1,1); \draw[kernels2] (0,0) node[not] {} -- (1,1) node[xies] {} ; 
\draw[kernels2] (-1,1)  {} -- (0,2) node[xies] {};
\draw[kernels2] (-1,1) node[not] {} -- (-1,2.5) node[xies] {};
}

\DeclareSymbol{Xi4cabc1}{0}{\draw (-1,1) -- (-2,2) node[xi] {};\draw[kernels2] (0,0)  -- (-1,1); \draw[kernels2] (0,0) node[not] {} -- (1,1) node[xic] {} ; 
\draw[kernels2] (-1,1)  {} -- (0,2) node[xic] {};
\draw[kernels2] (-1,1) node[not] {} -- (-1,2.5) node[xi] {};
}

\DeclareSymbol{Xi4cabc2}{0}{\draw (-1,1) -- (-2,2) node[xic] {};\draw[kernels2] (0,0)  -- (-1,1); \draw[kernels2] (0,0) node[not] {} -- (1,1) node[xic] {} ; 
\draw[kernels2] (-1,1)  {} -- (0,2) node[xi] {};
\draw[kernels2] (-1,1) node[not] {} -- (-1,2.5) node[xi] {};
}

\DeclareSymbol{Xi4ea}{1.5}{\draw (-1,2.5) node[xi] {} -- (-1,1) node[xi] {} -- (0,0); 
 \draw[kernels2] (0,0)  -- (1,1) node[xi] {};
\draw[kernels2] (0,0) node[not] {} -- (0,1.5) node[xi] {}; }

\DeclareSymbol{Xi4eac1}{1.5}{\draw (-1,2.5) node[xic] {} -- (-1,1) node[xi] {} -- (0,0); 
 \draw[kernels2] (0,0)  -- (1,1) node[xic] {};
\draw[kernels2] (0,0) node[not] {} -- (0,1.5) node[xi] {}; }

\DeclareSymbol{Xi4eac1b}{1.5}{\draw (-1,2.5) node[xic] {} -- (-1,1) node[xi] {} -- (0,0); 
 \draw[kernels2] (0,0)  -- (1,1) node[xiesf] {};
\draw[kernels2] (0,0) node[not] {} -- (0,1.5) node[xi] {}; }

\DeclareSymbol{Xi4eac2}{1.5}{\draw (-1,2.5) node[xic] {} -- (-1,1) node[xic] {} -- (0,0); 
 \draw[kernels2] (0,0)  -- (1,1) node[xi] {};
\draw[kernels2] (0,0) node[not] {} -- (0,1.5) node[xi] {}; }

\DeclareSymbol{Xi4eact1}{1.5}{\draw (-1,2.5) node[xic] {} -- (-1,1) node[xi] {} -- (0,0); 
 \draw (0,0)  -- (1,1) node[xic] {};
\draw[rho] (0,0) node[not] {} -- (0,1.5) node[xi] {}; }

\DeclareSymbol{Xi4eact2}{1.5}{\draw[rho] (-1,2.5) node[xic] {} -- (-1,1) node[xi] {} -- (0,0); 
 \draw (0,0)  -- (1,1) node[xic] {};
\draw (0,0) node[not] {} -- (0,1.5) node[xi] {}; }

\DeclareSymbol{Xi4eabis}{1.5}{\draw (-1,2.5) node[xi] {} -- (-1,1) ; \draw[kernels2] (-1,1) node[xi] {} -- (0,0); 
 \draw (0,0)  -- (1,1) node[xi] {};
\draw[kernels2] (0,0) node[not] {} -- (0,1.5) node[xi] {}; }

\DeclareSymbol{Xi4eabisc1}{1.5}{\draw (-1,2.5) node[xic] {} -- (-1,1) ; \draw[kernels2] (-1,1) node[xi] {} -- (0,0); 
 \draw (0,0)  -- (1,1) node[xi] {};
\draw[kernels2] (0,0) node[not] {} -- (0,1.5) node[xic] {}; }

\DeclareSymbol{Xi4eabisc1b}{1.5}{\draw (-1,2.5) node[xic] {} -- (-1,1) ; \draw[kernels2] (-1,1) node[xi] {} -- (0,0); 
 \draw (0,0)  -- (1,1) node[xi] {};
\draw[kernels2] (0,0) node[not] {} -- (0,1.5) node[xiesf] {}; }

\DeclareSymbol{Xi4eabisc1bis}{1.5}{\draw (-1,2.5) node[xi] {} -- (-1,1) ; \draw[kernels2] (-1,1) node[xi] {} -- (0,0); 
 \draw (0,0)  -- (1,1) node[xi] {};
\draw[kernels2] (0,0) node[not] {} -- (0,1.5) node[xi] {};
\draw (-2,1) node[] {\tiny{$i$}};
\draw (-2,2.5) node[] {\tiny{$\ell$}};
\draw (2,1) node[] {\tiny{$k$}};
\draw (0,2.5) node[] {\tiny{$j$}};
 }

\DeclareSymbol{Xi4eabisc1tris}{1.5}{\draw (-1,2.5) node[xi] {} -- (-1,1) ; \draw[kernels2] (-1,1) node[xi] {} -- (0,0); 
 \draw (0,0)  -- (1,1) node[xi] {};
\draw[kernels2] (0,0) node[not] {} -- (0,1.5) node[xi] {};
\draw (-2,1) node[] {\tiny{i}};
\draw (-2,2.5) node[] {\tiny{j}};
\draw (2,1) node[] {\tiny{j}};
\draw (0,2.5) node[] {\tiny{i}};
 }

\DeclareSymbol{Xi4eabisc1quater}{1.5}{\draw (-1,2.5) node[xic] {} -- (-1,1) ; \draw[kernels2] (-1,1) node[xi] {} -- (0,0); 
 \draw (0,0)  -- (1,1) node[xic] {};
\draw[kernels2] (0,0) node[not] {} -- (0,1.5) node[xi] {};
 }

\DeclareSymbol{Xi4eabisc2}{1.5}{\draw (-1,2.5) node[xic] {} -- (-1,1) ; \draw[kernels2] (-1,1) node[xi] {} -- (0,0); 
 \draw (0,0)  -- (1,1) node[xic] {};
\draw[kernels2] (0,0) node[not] {} -- (0,1.5) node[xi] {}; }

\DeclareSymbol{Xi4eabisc2l}{1.5}{\draw (-1,2.5) node[xiesf] {} -- (-1,1) ; \draw[kernels2] (-1,1) node[xi] {} -- (0,0); 
 \draw (0,0)  -- (1,1) node[xic] {};
\draw[kernels2] (0,0) node[not] {} -- (0,1.5) node[xi] {}; }

\DeclareSymbol{Xi4eabisc2r}{1.5}{\draw (-1,2.5) node[xic] {} -- (-1,1) ; \draw[kernels2] (-1,1) node[xi] {} -- (0,0); 
 \draw (0,0)  -- (1,1) node[xiesf] {};
\draw[kernels2] (0,0) node[not] {} -- (0,1.5) node[xi] {}; }

\DeclareSymbol{Xi4eabisc3}{1.5}{\draw (-1,2.5) node[xic] {} -- (-1,1) ; \draw[kernels2] (-1,1) node[xic] {} -- (0,0); 
 \draw (0,0)  -- (1,1) node[xi] {};
\draw[kernels2] (0,0) node[not] {} -- (0,1.5) node[xi] {}; }

\DeclareSymbol{Xi4eb}{0}{
\draw[kernels2] (0,2) node[xi] {} -- (-1,1) ; \draw[kernels2] (-2,2)  node[xi] {} -- (-1,1) ; \draw (-1,1)  node[not] {} -- (0,0); 
 \draw (0,0) node[xi] {}  -- (1,1) node[xi] {};
}

\DeclareSymbol{Xi4eab}{1.5}{\draw[kernels2] (-1,2.5) node[xi] {} -- (-1,1) ; \draw[kernels2] (-2,2)  node[xi] {} -- (-1,1) ; \draw (-1,1)  node[not] {} -- (0,0); 
 \draw[kernels2] (0,0)  -- (1,1) node[xi] {};
\draw[kernels2] (0,0) node[not] {} -- (0,1.5) node[xi] {}; 
}

\DeclareSymbol{Xi4eabdis}{1.5}{\draw[kernels2] (-1,2.5) node[xies] {} -- (-1,1) ; \draw[kernels2] (-2,2)  node[xies] {} -- (-1,1) ; \draw (-1,1)  node[not] {} -- (0,0); 
 \draw[kernels2] (0,0)  -- (1,1) node[xies] {};
\draw[kernels2] (0,0) node[not] {} -- (0,1.5) node[xies] {}; 
}

\DeclareSymbol{Xi4eabc1}{1.5}{\draw[kernels2] (-1,2.5) node[xic] {} -- (-1,1) ; \draw[kernels2] (-2,2)  node[xi] {} -- (-1,1) ; \draw (-1,1)  node[not] {} -- (0,0); 
 \draw[kernels2] (0,0)  -- (1,1) node[xic] {};
\draw[kernels2] (0,0) node[not] {} -- (0,1.5) node[xi] {}; 
}

\DeclareSymbol{Xi4eabc2}{1.5}{\draw[kernels2] (-1,2.5) node[xi] {} -- (-1,1) ; \draw[kernels2] (-2,2)  node[xi] {} -- (-1,1) ; \draw (-1,1)  node[not] {} -- (0,0); 
 \draw[kernels2] (0,0)  -- (1,1) node[xic] {};
\draw[kernels2] (0,0) node[not] {} -- (0,1.5) node[xic] {}; 
}

\DeclareSymbol{Xi4eabbis}{1.5}{\draw[kernels2] (-1,2.5) node[xi] {} -- (-1,1) ; \draw[kernels2] (-2,2)  node[xi] {} -- (-1,1) ; \draw[kernels2] (-1,1)  node[not] {} -- (0,0); 
 \draw (0,0)  -- (1,1) node[xi] {};
\draw[kernels2] (0,0) node[not] {} -- (0,1.5) node[xi] {}; 
}

\DeclareSymbol{Xi4eabbisc1}{1.5}{\draw[kernels2] (-1,2.5) node[xic] {} -- (-1,1) ; \draw[kernels2] (-2,2)  node[xi] {} -- (-1,1) ; \draw[kernels2] (-1,1)  node[not] {} -- (0,0); 
 \draw (0,0)  -- (1,1) node[xic] {};
\draw[kernels2] (0,0) node[not] {} -- (0,1.5) node[xi] {}; 
}

\DeclareSymbol{Xi4eabbisc1perm}{1.5}{\draw[kernels2] (-1,2.5) node[xi] {} -- (-1,1) ; \draw[kernels2] (-2,2)  node[xic] {} -- (-1,1) ; \draw[kernels2] (-1,1)  node[not] {} -- (0,0); 
 \draw (0,0)  -- (1,1) node[xic] {};
\draw[kernels2] (0,0) node[not] {} -- (0,1.5) node[xi] {}; 
}

\DeclareSymbol{Xi4eabbisc2}{1.5}{\draw[kernels2] (-1,2.5) node[xi] {} -- (-1,1) ; \draw[kernels2] (-2,2)  node[xi] {} -- (-1,1) ; \draw[kernels2] (-1,1)  node[not] {} -- (0,0); 
 \draw (0,0)  -- (1,1) node[xic] {};
\draw[kernels2] (0,0) node[not] {} -- (0,1.5) node[xic] {}; 
}

\DeclareSymbol{Xi2cbis}{0}{\draw[kernels2] (0,1) -- (0.8,2.2) node[xi] {};\draw[kernels2] (0,-0.25) node[not] {} -- (0,1); \draw[kernels2] (0,1) node[not] {} -- (-0.8,2.2) node[xi] {};}

\DeclareSymbol{Xi2cbis1}{0}{\draw (0,1) -- (-0.8,2.2) node[xi] {};\draw[kernels2] (0,-0.25) node[not] {} -- (0,1) node[xi] {}; }

\DeclareSymbol{Xi2Xbis}{-2}{\draw[kernels2] (0,-0.25)  -- (-1,1) ; \draw (-1,1) node[xix] {};
\draw[kernels2] (0,-0.25) node[not] {} -- (1,1) node[xi] {};}

\DeclareSymbol{XXi2bis}{-2}{\draw[kernels2] (0,-0.25) -- (-1,1) node[xi] {};
\draw[kernels2] (0,-0.25) node[X] {} -- (1,1) node[xi] {};}

\DeclareSymbol{I1XiIXi}{0}{\draw[kernels2] (0,-0.25) -- (1,1) node[xi] {};
\draw (0,-0.25) node[not] {} -- (-1,1) node[xi] {};}

\DeclareSymbol{I1XiIXib}{0}{\draw  (0,-0.25) node[xi] {} -- (0,1) node[not] {};
\draw[kernels2] (0,1) -- (0,2.25) ; \draw (0,2.25) node[xi]{}; }

\DeclareSymbol{I1XiIXic}{0}{
\draw[kernels2] (0,0) -- (1,1) node[xi] {} ; 
\draw[kernels2] (0,0) node[not] {}  -- (-1,1) node[not] {} -- (0,2) node[xi] {};
}

\DeclareSymbol{thin}{1.4}{\draw[pagebackground] (-0.3,0) -- (0.3,0); \draw  (0,0) -- (0,2);}
\DeclareSymbol{thin2}{1.4}{\draw[pagebackground] (-0.3,0) -- (0.3,0); \draw[tinydots]  (0,0) -- (0,2);}

\DeclareSymbol{thick}{1.4}{\draw[pagebackground] (-0.3,0) -- (0.3,0); \draw[kernels2]  (0,0) -- (0,2);}

\DeclareSymbol{thick2}{1.4}{\draw[pagebackground] (-0.3,0) -- (0.3,0); \draw[kernels2,tinydots]  (0,0) -- (0,2);}

\DeclareSymbol{Xi4ind}{2}{\draw (0,0) node[xi,label={[label distance=-0.2em]right: \scriptsize  $ i $}]  { } -- (-1,1) node[xi,label={[label distance=-0.2em]left: \scriptsize  $ j $}] {} -- (0,2) node[xi,label={[label distance=-0.2em]right: \scriptsize  $ k $}] {} -- (-1,3) node[xi,label={[label distance=-0.2em]left: \scriptsize  $ \ell $}] {};}

\DeclareSymbol{Xi4c1}{2}{\draw (0,0) node[xic] {} -- (-1,1) node[xi] {} -- (0,2) node[xic] {} -- (-1,3) node[xi] {};} 
\DeclareSymbol{IXi2ex}{0}{\draw (0,-0.25) node[xie] {} -- (-1,1) node[xi] {} ; \draw (0,-0.25)-- (1,1) node[xi] {};}
\DeclareSymbol{IXi2ex1}{0}{\draw (0,-0.25) node[xie] {} -- (-1,1) node[xi] {} -- (0,2) node[xi] {};}

\DeclareSymbol{Xi4b1}{0}{\draw(0,1.5) node[xic] {} -- (0,0); \draw (-1,1) node[xic] {} -- (0,0) node[xi] {} -- (1,1) node[xi] {};}

\DeclareSymbol{Xi4ec1}{0}{\draw (0,2) node[xi] {} -- (-1,1) node[xic] {} -- (0,0) node[xic] {} -- (1,1) node[xi] {};}
\DeclareSymbol{Xi4ec2}{0}{\draw (0,2) node[xic] {} -- (-1,1) node[xi] {} -- (0,0) node[xic] {} -- (1,1) node[xi] {};}
\DeclareSymbol{Xi4ec3}{0}{\draw (0,2) node[xic] {} -- (-1,1) node[xic] {} -- (0,0) node[xi] {} -- (1,1) node[xi] {};}

\DeclareSymbol{I1Xi4ac1}{2}{\draw[kernels2] (0,0) node[not] {} -- (-1,1) ; \draw[kernels2] (0,0) node[not] {} -- (1,1) node[xic] {} ;
\draw (-1,1) node[xi] {} -- (0,2) node[xic] {} -- (-1,3) node[xi] {};}

\DeclareSymbol{I1Xi4ac2}{2}{\draw[kernels2] (0,0) node[not] {} -- (-1,1) ; \draw[kernels2] (0,0) node[not] {} -- (1,1) node[xic] {} ;
\draw (-1,1) node[xi] {} -- (0,2) node[xi] {} -- (-1,3) node[xic] {};}

\DeclareSymbol{I1Xi4bp}{2}{\draw (0,0) node[not] {} -- (-1,1) node[xi] {} -- (0,2) ; \draw[kernels2] (0,2) node[not] {} -- (-1,3) node[xi] {};\draw[kernels2] (0,2)  -- (1,3) node[xi] {};
}

\DeclareSymbol{I1Xi4bc1}{2}{\draw (0,0) node[xic] {} -- (-1,1) node[xi] {} -- (0,2) ; \draw[kernels2] (0,2) node[not] {} -- (-1,3) node[xi] {};\draw[kernels2] (0,2)  -- (1,3) node[xic] {};
}

\DeclareSymbol{I1Xi4bc2}{2}{\draw (0,0) node[xic] {} -- (-1,1) node[xi] {} -- (0,2) ; \draw[kernels2] (0,2) node[not] {} -- (-1,3) node[xic] {};\draw[kernels2] (0,2)  -- (1,3) node[xi] {};
}

\DeclareSymbol{I1Xi4cp}{2}{\draw (0,0) node[not] {} -- (-1,1) node[not] {}; \draw[kernels2] (-1,1) -- (0,2) ; 
\draw[kernels2] (-1,1) -- (-2,2) node[xi] {} ;
\draw (0,2) node[xi] {} -- (-1,3) node[xi] {};}

\DeclareSymbol{I1Xi4cc1}{2}{\draw (0,0) node[xic] {} -- (-1,1) node[not] {}; \draw[kernels2] (-1,1) -- (0,2) ; 
\draw[kernels2] (-1,1) -- (-2,2) node[xi] {} ;
\draw (0,2) node[xic] {} -- (-1,3) node[xi] {};}

\DeclareSymbol{I1Xi4cc2}{2}{\draw (0,0) node[xic] {} -- (-1,1) node[not] {}; \draw[kernels2] (-1,1) -- (0,2) ; 
\draw[kernels2] (-1,1) -- (-2,2) node[xi] {} ;
\draw (0,2) node[xi] {} -- (-1,3) node[xic] {};}

\DeclareSymbol{I1Xi4abc1}{2}{\draw[kernels2] (0,0) node[not] {} -- (-1,1) ; \draw[kernels2] (0,0) node[not] {} -- (1,1) node[xic] {};\draw (-1,1) node[xi] {} -- (0,2) ; \draw[kernels2] (0,2) node[not] {} -- (-1,3) node[xic] {};\draw[kernels2] (0,2)  -- (1,3) node[xi] {}; }

\DeclareSymbol{I1Xi4abc2}{2}{\draw[kernels2] (0,0) node[not] {} -- (-1,1) ; \draw[kernels2] (0,0) node[not] {} -- (1,1) node[xic] {};\draw (-1,1) node[xi] {} -- (0,2) ; \draw[kernels2] (0,2) node[not] {} -- (-1,3) node[xi] {};\draw[kernels2] (0,2)  -- (1,3) node[xic] {}; }

\DeclareSymbol{R1}{0}{\draw (-1,1) node[xi] {} -- (0,0) node[not] {};
\draw[kernels2] (0,1.5) node[xic] {} -- (0,0) -- (1,1) node[xic] {};}
\DeclareSymbol{R2}{0}{\draw (-1,1) node[xic] {} -- (0,0) node[not] {};
\draw[kernels2] (0,1.5)  {} -- (0,0) -- (1,1)  {};
\draw (0,1.5) node[xi] {};
\draw (1,1) node[xic] {};
}
\DeclareSymbol{R3}{1}{\draw[kernels2] (-1,1.5)  {} -- (0,0) node[not] {} -- (1,1.5);
\draw (-1,1.5) node[xi] {};
\draw[kernels2] (0,3) {} -- (1,1.5) -- (2,3)  {};
\draw  (0,3) node[xic] {} ;
\draw (2,3) node[xic] {};}
\DeclareSymbol{R4}{1}{\draw[kernels2] (-1,1.5) node[xic] {} -- (0,0) node[not] {} -- (1,1.5);
\draw[kernels2] (0,3) {} -- (1,1.5) -- (2,3) node[xic] {};
\draw (0,3) node[xi] {};}

\DeclareSymbol{I1Xi4bcp}{2}{\draw (0,0) node[not] {} -- (-1,1) node[not] {}; \draw[kernels2] (-1,1) -- (0,2) ; 
\draw[kernels2] (-1,1) -- (-2,2) node[xi] {} ; \draw[kernels2] (0,2) node[not] {} -- (-1,3) node[xi] {};\draw[kernels2] (0,2)  -- (1,3) node[xi] {};
}

\DeclareSymbol{I1Xi4bcc1}{2}{\draw (0,0) node[xic] {} -- (-1,1) node[not] {}; \draw[kernels2] (-1,1) -- (0,2) ; 
\draw[kernels2] (-1,1) -- (-2,2) node[xi] {} ; \draw[kernels2] (0,2) node[not] {} -- (-1,3) node[xi] {};\draw[kernels2] (0,2)  -- (1,3) node[xic] {};
}

\DeclareSymbol{I1Xi4bcc2}{2}{\draw (0,0) node[xic] {} -- (-1,1) node[not] {}; \draw[kernels2] (-1,1) -- (0,2) ; 
\draw[kernels2] (-1,1) -- (-2,2) node[xi] {} ; \draw[kernels2] (0,2) node[not] {} -- (-1,3) node[xic] {};\draw[kernels2] (0,2)  -- (1,3) node[xi] {};
} 

\DeclareSymbol{2I1Xi4bc1}{2}{\draw[kernels2] (0,0) node[not] {} -- (-1,1) ;
\draw[kernels2] (0,0) -- (1,1);
\draw (-1,1) node[xic] {} -- (-1,2.5) node[xi] {};
\draw (1,1)  node[xic] {} -- (1,2.5) node[xi] {};
}

\DeclareSymbol{2I1Xi4bc2}{2}{\draw[kernels2] (0,0) node[not] {} -- (-1,1) ;
\draw[kernels2] (0,0) -- (1,1);
\draw (-1,1) node[xi] {} -- (-1,2.5) node[xic] {};
\draw (1,1)  node[xic] {} -- (1,2.5) node[xi] {};
}

\DeclareSymbol{diff2I1Xi4bc2}{2}{\draw (0,0) node[diff] {} -- (-1,1) ;
\draw (0,0) -- (1,1);
\draw (-1,1) node[xi] {} -- (-1,2.5) node[xic] {};
\draw (1,1)  node[xic] {} -- (1,2.5) node[xi] {};
}

\DeclareSymbol{2I1Xi4bc3}{2}{\draw[kernels2] (0,0) node[not] {} -- (-1,1) ;
\draw[kernels2] (0,0) -- (1,1);
\draw (-1,1) node[xic] {} -- (-1,2.5) node[xic] {};
\draw (1,1)  node[xi] {} -- (1,2.5) node[xi] {};
}

\DeclareSymbol{Xi41}{0}{\draw (0,1) -- (0.8,2.2) node[xic] {};\draw (0,-0.25) node[xi] {} -- (0,1) node[xi] {} -- (-0.8,2.2) node[xic] {};} 

\DeclareSymbol{Xi42}{0}{\draw (0,1) -- (0.8,2.2) node[xi] {};\draw (0,-0.25) node[xic] {} -- (0,1) node[xi] {} -- (-0.8,2.2) node[xic] {};}

\DeclareSymbol{Xi4ca1}{0}{\draw (0,1) -- (-1,2.2) node[xic] {};\draw (0,-0.25) node[xi] {} -- (0,1) ; \draw[kernels2] (0,1) node[not] {} -- (1,2.2) node[xic] {};
\draw[kernels2] (0,1) {} -- (0,2.7) node[xi] {};
}

\DeclareSymbol{Xi4ca2}{0}{\draw (0,1) -- (-1,2.2) node[xi] {};\draw (0,-0.25) node[xi] {} -- (0,1) ; \draw[kernels2] (0,1) node[not] {} -- (1,2.2) node[xic] {};
\draw[kernels2] (0,1) {} -- (0,2.7) node[xic] {};
}

\DeclareSymbol{Xi4cap}{0}{\draw (0,1) -- (-1,2.2) node[xi] {};\draw (0,-0.25) node[not] {} -- (0,1) ; \draw[kernels2] (0,1) node[not] {} -- (1,2.2) node[xi] {};
\draw[kernels2] (0,1) {} -- (0,2.7) node[xi] {};
}

\DeclareSymbol{Xi3a}{0}{
 \draw (-1,1)  node[xi] {} -- (0,0); 
 \draw (0,0) node[xi] {}  -- (1,1) node[xi] {};
}

\DeclareSymbol{Xi4ebc1}{0}{
\draw[kernels2] (0,2) node[xi] {} -- (-1,1) ; \draw[kernels2] (-2,2)  node[xic] {} -- (-1,1) ; \draw (-1,1)  node[not] {} -- (0,0); 
 \draw (0,0) node[xic] {}  -- (1,1) node[xi] {};
}

\DeclareSymbol{Xi4ebc2}{0}{
\draw[kernels2] (0,2) node[xi] {} -- (-1,1) ; \draw[kernels2] (-2,2)  node[xi] {} -- (-1,1) ; \draw (-1,1)  node[not] {} -- (0,0); 
 \draw (0,0) node[xic] {}  -- (1,1) node[xic] {};
}

\DeclareSymbol{Xi2cbispex}{0}{\draw[kernels2] (0,1) -- (0.8,2.2) node[xi] {};\draw (0,-0.25) node[xie] {} -- (0,1); \draw[kernels2] (0,1) node[not] {} -- (-0.8,2.2) node[xi] {};}

\DeclareSymbol{Xi2cbis1p}{0}{\draw (0,1) -- (-0.8,2.2) node[xi] {};\draw (0,-0.25) node[not] {} -- (0,1) node[xi] {}; }

\DeclareSymbol{Xi2Xp}{-2}{\draw (0,-0.25) node[not] {} -- (-1,1) node[xix] {};} 

\DeclareSymbol{I1XiIXib}{0}{\draw  (0,-0.25) node[xi] {} -- (0,1) node[not] {};
\draw[kernels2] (0,1) -- (0,2.25) ; \draw (0,2.25) node[xi]{}; }

\DeclareSymbol{IXi2b}{0}{\draw  (0,-0.25) node[xi] {} -- (0,1) node[not] {};
\draw (0,1) -- (0,2.25) ; \draw (0,2.25) node[xi]{}; }

\DeclareSymbol{IXi2bex}{0}{\draw  (0,-0.25) node[xi] {} -- (0,1) node[xie] {};
\draw (0,1) -- (0,2.25) ; \draw (0,2.25) node[xi]{}; }

 \def\1{\mathbf{\symbol{1}}}

\def\one{\mathbf{1}}

\DeclareSymbol{diff}{0}{
\draw (0,0.5) node[diff] {};
}

\DeclareSymbol{diff1}{0}{
\draw (0,0.5) node[diff1] {};
}

\DeclareSymbol{diff2}{0}{
\draw (0,0.5) node[diff2] {};
}

\DeclareSymbol{geo}{0}{
\draw (0,0) node[diff] {};
\draw (0.3,0) node[diff] {};
}

\DeclareSymbol{generic}{0}{
\draw (0,0.6) node[xi] {};
}

\DeclareSymbol{g}{0}{
\draw (0,0.6) node[g] {};
}

\DeclareSymbol{Ito}{0}{
\draw (0,0.6) node[xies] {};
}

\DeclareSymbol{Itob}{0}{
\draw (0,0.6) node[xiesf] {};
}

\DeclareSymbol{greycirc}{0}{
\draw (0,0.3) node[xi] {};
}

\DeclareSymbol{not}{0}{
\draw (0,0.6) node[not] {};
\draw[tinydots] (0,0.6) circle (0.8);
}

\DeclareSymbol{genericb}{0}{
\draw (0,0.6) node[xic] {};
}

\DeclareSymbol{bluecirc}{0}{
\draw (0,0.3) node[xic] {};
}

\DeclareSymbol{genericxix}{0}{
\draw (0,0.6) node[xix] {};
}

\DeclareSymbol{genericX}{0}{
\draw (0,0.6) node[X] {};
}

\DeclareSymbol{diffIto}{1}{
\draw  (0,2.5) -- (0,0) ;
\draw (0,-0.1) node[diff] {};
\draw (0,2.5) node[xies] {};
}
\DeclareSymbol{Itodiff}{2}{
\draw(0,2.9) -- (0,-0.2);
\draw (0,2.9) node[diff] {};
\draw (0,-0.1) node[xies] {};
}

\DeclareSymbol{diffgeneric}{1}{
\draw  (0,2.5) -- (0,0) ;
\draw (0,-0.1) node[diff] {};
\draw (0,2.5) node[xi] {};
}

\DeclareSymbol{genericdiff}{2}{
\draw(0,2.9) -- (0,-0.2);
\draw (0,2.9) node[diff] {};
\draw (0,-0.1) node[xi] {};
}

\DeclareSymbol{diffdot}{2}{
\draw  (0,3) -- (0,-0.1) ;
\draw (0,3) node[not] {};
\draw (0,-0.1) node[diff] {};
}

\DeclareSymbol{diffdotmini}{0}{
\draw  (0,0) -- (0,1.2) ;
\draw (0,1.2) node[not] {};
\draw (0,0) node[diffmini] {};
}

\DeclareSymbol{dotdiff}{2}{
\draw[kernelsmod]  (0,3) -- (0,-0.1) ;
\draw (0,3) node[diff] {};
\draw (0,-0.1) node[not] {};
}

\DeclareSymbol{dotdiff1}{2}{
\draw[kernelsmod]  (0,3) -- (0,-0.1) ;
\draw (0,3) node[diff1] {};
\draw (0,-0.1) node[not] {};
}

\DeclareSymbol{dotdiff1mini}{0}{
\draw[kernelsmod]  (0,1.2) -- (0,0) ;
\draw (0,1.2) node[diffmini] {};
\draw (0,0) node[not] {};
}

\DeclareSymbol{dotdiff2}{2}{
\draw (0,3) -- (0,-0.1) ;
\draw (0,3) node[diff] {};
\draw (0,-0.1) node[not] {};
}

\DeclareSymbol{dotdiff2mini}{0}{
\draw (0,1.2) -- (0,0) ;
\draw (0,1.2) node[diffmini] {};
\draw (0,0) node[not] {};
}

\DeclareSymbol{dotdiffstraight}{0}{
\draw  (0,3) -- (0,-0.1) ;
\draw (0,3) node[diff] {};
\draw (0,-0.1) node[not] {};
}

\DeclareSymbol{arbre1}{0}{
\draw  (0,0) -- (1.5,1.5) ;
\draw (1.5,1.5) node[not] {};
\draw (0,0) node[not] {};
}

\DeclareSymbol{arbre2}{0}{
\draw  (0,0) -- (1.5,1.5) ;
\draw[kernelsmod] (0,0) -- (-1.5,1.5);
\draw (1.5,1.5) node[not] {};
\draw (0,0) node[not] {};
\draw (-1.5,1.5) node[xi] {};
}

\DeclareSymbol{arbre3}{0}{
\draw  (0,0) -- (1.5,1.5) ;
\draw[kernelsmod] (1.5,1.5) -- (0,3);
\draw (0,0) node[not] {};
\draw (1.5,1.5) node[not] {};
\draw (0,3) node[xi] {};
}

\DeclareSymbol{treeeval}{0}{
\draw (0,0) -- (1,1);
\draw (0,0) node[xi] {};
\draw (1.25,1.25) node[xi] {};
\draw (-0.6,0.6) node[]{\tiny{$i$}};
\draw (0.65,1.85) node[]{\tiny{$j$}};
}

\DeclareSymbol{testeval}{0}{
\draw (0,0) -- (1,1);
\draw (0,0) -- (-1,1);
\draw (0,0) node[xi] {};
\draw (1.25,1.25) node[xi] {};
\draw (-1.25,1.25) node[xi] {};
\draw (-0.6,-0.6) node[]{\tiny{$i$}};
\draw (0.65,1.85) node[]{\tiny{$j$}};
\draw (-1.95,1.85) node[]{\tiny{$k$}};
}

\DeclareSymbol{treeeval2}{0}{
\draw[kernelsmod] (-0.25,-1) -- (1,0.5) ;
\draw[kernelsmod] (1,0.5) -- (-0.25,2);
\draw (1,0.5) node[diff2] {};
\draw (-0.25,-1) node[not] {};
\draw (-0.25,2) node[xi] {};
\draw (-0.6,1.2) node[]{\tiny{1}};
}

\DeclareSymbol{arbreact}{1}{
\draw (0,0) node[not] {};
\draw[kernelsmod] (0,0) -- (1,1);
\draw[kernelsmod] (0,0) -- (-1,1);
\draw (-1,1) node[xic] {};
\draw  (0,2) -- (1,1) ;
\draw (0,2) node[xic] {};
\draw (1,1) node[xi] {};
}

\DeclareSymbol{arbreact1}{0}{
\draw (0,-1.5) -- (0,0);
\draw[kernelsmod] (0,0) -- (1,1);
\draw[kernelsmod] (0,0) -- (-1,1);
\draw  (0,2) -- (1,1) ;
\draw (0,-1.5) node[diff] {};
\draw (0,0) node[not] {};
\draw (-1,1) node[xic] {};
\draw (0,2) node[xic] {};
\draw (1,1) node[xi] {};
}

\DeclareSymbol{arbreact2}{0}{
\draw (0,-0.75) -- (-1,0.5); 
\draw (0,-0.75) -- (1,0.5);
\draw (0,1.5) -- (1,0.5);
\draw (0,1.5) node[xic] {};
\draw (1,0.5) node[xi] {};
\draw (-1,0.5) node[xic] {};
\draw (0,-0.75) node[diff] {};
}

\DeclareSymbol{arbreact3}{0}{
\draw[kernelsmod] (0,-0.75) -- (-1,0.5); 
\draw[kernelsmod] (0,-0.75) -- (1,0.5);
\draw (0,1.75) -- (1,0.5);
\draw (2,1.75) -- (1,0.5);
\draw (0,1.75) node[xic] {};
\draw (1,0.5) node[diff] {};
\draw (-1,0.5) node[xic] {};
\draw (2,1.75) node[xi] {};
\draw (0,-0.75) node[not] {};
}

\DeclareSymbol{pre_im_I1Xitwo}{0}{
\draw[kernels2] (0,-0.3) node[not] {} -- (-0.6,0.7) ;
\draw[kernels2] (0,-0.3) -- (0.6,0.7);
\draw (0,0.9) node[g] {};
}

\DeclareSymbol{pre_im_cI1Xi4ab}{2}{
\draw[kernels2] (0,-1) node[not] {} -- (-0.6,0) ;
\draw[kernels2] (0,-1) -- (0.6,0);
\draw (0,0.2) node[g] {};
\draw (0,0.6) -- (0,1.5);
\draw[kernels2] (0,1.5) node[not] {} -- (-0.6,2.5) ;
\draw[kernels2] (0,1.5) -- (0.6,2.5);
\draw (0,2.7) node[g] {};
}

\DeclareSymbol{pre_im_I1Xi4acc2}{0}{
\draw[kernels2] (-1,-0.5) node[not] {} -- (-1.6,0.5) ;
\draw[kernels2] (-1,-0.5) -- (-0.4,0.5);
\draw[kernels2] (-1,-0.5) -- (0.2,-1.5) node[not] {} ;
\draw (-1,1.1) -- (-1,2);
\draw[kernels2] (0.2,-1.5) -- (0.2,2);
\draw (-1,0.7) node[g] {};
\draw (-0.3,2.2) node[g] {};
}

\DeclareSymbol{pre_im_I1Xi4abcc2}{2}{
\draw[kernels2] (0,-1) node[not] {} -- (-1,0) node[not] {};
\draw[kernels2] (-1,1.2) node[not] {} -- (-1,0);
\draw[kernels2] (-1,1.2) -- (-1.5,2.5);
\draw[kernels2] (-1,1.2) -- (-0.5,2.5);
\draw[kernels2] (-1,0) -- (0.7,2.5);
\draw[kernels2] (0,-1) -- (1.5,2.5);
\draw (-1,2.7) node[g] {};
\draw (1,2.7) node[g] {};
}

\DeclareSymbol{pre_im_2I1Xi4c1}{2}{
\draw[kernels2] (0,-0.5) node[not] {} -- (-1,0.5) node[not] {};
\draw[kernels2] (0,-0.5) -- (1,0.5) node[not] {};
\draw[kernels2] (-1,0.5) node[not] {}-- (-1.7,2);
\draw[kernels2]  (-1,2) -- (1,0.5);
\draw[kernels2] (-1,0.5) -- (1,2);
\draw[kernels2] (1,0.5) -- (1.7,2);
\draw (-1.2,2.2) node[g] {};
\draw (1.2,2.2) node[g] {};
}

\DeclareSymbol{pre_im_Xi4eabisc2}{2}{
\draw[kernels2] (1.2,-0.5) node[not] {} -- (-0.7,0.8) ;
\draw[kernels2] (1.2,-0.5) -- (0.4,0.8);
\draw (0,1.4)  -- (0,2.2);
\draw (1.2,2.2) -- (1.2,-0.6);
\draw (0,1) node[g] {};
\draw (0.6,2.4) node[g] {};
}

\DeclareSymbol{pre_im_Xi4eabisc22}{2}{
\draw (1.2,-0.5) node[not] {} -- (-0.7,0.8) ;
\draw[kernels2] (1.2,-0.5) -- (0.4,0.8);
\draw (0,1.4)  -- (0,2.2);
\draw[kernels2] (1.2,2.2) -- (1.2,-0.6);
\draw (0,1) node[g] {};
\draw (0.6,2.4) node[g] {};
}

\DeclareSymbol{pre_im_Xi4eabisc222}{2}{
\draw[kernels2] (0.4,-0.5) node[not] {} -- (-0.6,1) ;
\draw[kernels2] (1.2,0) -- (0.3,1);
\draw (0,1.1)  -- (0,2.5);
\draw[kernels2] (1.2,2.5) -- (1.2,0) node[not] {} -- (0.4,-0.6);
\draw (0,1.2) node[g] {};
\draw (0.6,2.5) node[g] {};
}

\DeclareSymbol{pre_im_Xi4eabc2}{2}{
\draw (0,-0.5) node[not] {} -- (-1,0.5) node[not] {};
\draw[kernels2] (-1,0.5) -- (-1.5,2);
\draw[kernels2] (0,-0.5)  -- (0.7,2);
\draw[kernels2] (-1,0.5) -- (-0.5,2);
\draw[kernels2] (0,-0.5) -- (1.5,2);
\draw (-1,2.2) node[g] {};
\draw (1,2.2) node[g] {};
}

\DeclareSymbol{pre_im_Xi4eabbisc2}{2}{
\draw[kernels2] (0,-0.5) node[not] {} -- (-1,0.5) node[not] {};
\draw[kernels2] (-1,0.5) -- (-1.5,2);
\draw[kernels2] (0,-0.5)  -- (0.7,2);
\draw[kernels2] (-1,0.5) -- (-0.5,2);
\draw (0,-0.5) -- (1.5,2);
\draw (-1,2.2) node[g] {};
\draw (1,2.2) node[g] {};
}

\DeclareSymbol{pre_im_I1Xi4abcc1}{2}{
\draw[kernels2] (0,-1) node[not] {} -- (-1,0) node[not] {};
\draw[kernels2] (0,1.1) node[not] {} -- (-1,0);
\draw[kernels2] (-1,0) -- (-1.5,2.5);
\draw[kernels2] (0,1.1) node[not] {} -- (-0.5,2.5);
\draw[kernels2] (0,1.1) -- (0.5,2.5);
\draw[kernels2] (0,-1) -- (1.5,2.5);
\draw (-1,2.7) node[g] {};
\draw (1,2.7) node[g] {};
}

\DeclareSymbol{pre_im_Xi4eabc1}{2}{
\draw (0,-0.5) node[not] {} -- (-1,0.5) node[not] {};
\draw[kernels2] (-1,0.5) -- (-1.5,2);
\draw[kernels2] (0,-0.5)  -- (-0.5,2);
\draw[kernels2] (-1,0.5) -- (0.8,2);
\draw[kernels2] (0,-0.5) -- (1.5,2);
\draw (-1,2.2) node[g] {};
\draw (1,2.2) node[g] {};
}

\DeclareSymbol{pre_im_Xi4ba1b}{2}{
\draw[kernels2] (0,0) node[not] {}  -- (1.8,1.5);
\draw[kernels2] (0,0) -- (0.8,1.5);
\draw (0,-0.1) -- (-1.8,1.5);
\draw (0,-0.1) -- (-0.8,1.5);
\draw (-1,1.7) node[g] {};
\draw (1,1.7) node[g] {};
}

\DeclareSymbol{pre_im_Xi4ba2}{2}{
\draw (0,-0.1) node[not] {}  -- (1.8,1.5);
\draw[kernels2] (0,0) -- (0.8,1.5);
\draw (0,-0.1) -- (-1.8,1.5);
\draw[kernels2] (0,0) -- (-0.8,1.5);
\draw (-1,1.7) node[g] {};
\draw (1,1.7) node[g] {};
}

\DeclareSymbol{pre_im_Xi4cabc2}{2}{
\draw[kernels2] (0,-0.5) node[not] {} -- (-1,0.5) node[not] {};
\draw[kernels2] (-1,0.5) -- (-1.5,2);
\draw (-1,0.5)  -- (0.7,2);
\draw[kernels2] (-1,0.5) -- (-0.5,2);
\draw[kernels2] (0,-0.5) -- (1.5,2);
\draw (-1,2.2) node[g] {};
\draw (1,2.2) node[g] {};
}

\DeclareSymbol{pre_im_Xi4cabc1}{2}{
\draw[kernels2] (0,-0.5) node[not] {} -- (-1,0.5) node[not] {};
\draw (-1,0.5) -- (-1.5,2);
\draw[kernels2] (-1,0.5)  -- (0.7,2);
\draw[kernels2] (-1,0.5) -- (-0.5,2);
\draw[kernels2] (0,-0.5) -- (1.5,2);
\draw (-1,2.2) node[g] {};
\draw (1,2.2) node[g] {};
}

\DeclareSymbol{pre_im_Xi4eabbisc1}{2}{
\draw[kernels2] (0,-0.5) node[not] {} -- (-1,0.5) node[not] {};
\draw[kernels2] (-1,0.5) -- (-1.5,2);
\draw[kernels2] (0,-0.5)  -- (-0.5,2);
\draw[kernels2] (-1,0.5) -- (0.8,2);
\draw (0,-0.5) -- (1.5,2);
\draw (-1,2.2) node[g] {};
\draw (1,2.2) node[g] {};
}

\DeclareSymbol{pre_im_1}{0}{
\draw[kernels2] (0,-0.5) node[not] {} -- (-0.6,0.5) ;
\draw[kernels2] (0,-0.5) -- (0.6,0.5);
\draw (0,1.1)  -- (-0.55,2);
\draw (0,1.1)  -- (0.55,2);
\draw (0,0.7) node[g] {};
\draw (0,2.2) node[g] {};
}

\DeclareSymbol{disconnect}{0}{
\draw[kernels2] (0,-0.5) node[not] {} -- (-0.6,0.5) ;
\draw[kernels2] (0,-0.5) -- (0.6,0.5);
\draw (-0.55,1.1)  -- (-0.55,2.3);
\draw (0.55,2.3) -- (0.55,1.5) -- (1.2,1.5) -- (1.2,3.5) -- (0.55,3.5) -- (0.55,2.7);
\draw (0,0.7) node[g] {};
\draw (0,2.5) node[g] {};
}

\DeclareSymbol{pre_im_2}{2}{\draw[kernels2] (0,0) node[not] {} -- (-1,1) node[not] {};
\draw[kernels2] (0,0) -- (1,1) node[not] {};
\draw[kernels2] (-1,1) -- (-1.5,2.5);
\draw[kernels2] (-1,1) -- (-0.5,2.5);
\draw[kernels2] (1,1) -- (0.5,2.5);
\draw[kernels2] (1,1) -- (1.5,2.5);
\draw (-1,2.7) node[g] {};
\draw (1,2.7) node[g] {};
}

\DeclareSymbol{CX_rec}{0}{
\draw [black] (-0.3,1) to (-0.3,-0.3);
\draw [black] (0.3,1) to (0.3,-0.3);
\draw [black] (-0.3,1) to (-0.3,2.3);
\draw [black] (0.3,1) to (0.3,2.3);
\draw (0,1) node[rec] {};
}

\DeclareSymbol{CX_cerc}{0}{
\draw [black] (0,1) to (0,-0.3);
\draw (0,1) node[cerc] {};
}

\DeclareMathAlphabet{\mathpzc}{OT1}{pzc}{m}{it}

\def\eqref#1{(\ref{#1})}

\makeatletter 
\newcommand*{\bigcdot}{}
\DeclareRobustCommand*{\bigcdot}{%
  \mathbin{\mathpalette\bigcdot@{}}%
}
\newcommand*{\bigcdot@scalefactor}{.5}
\newcommand*{\bigcdot@widthfactor}{1.15}
\newcommand*{\bigcdot@}[2]{%
  \sbox0{$#1\vcenter{}$}
  \sbox2{$#1\cdot\m@th$}%
  \hbox to \bigcdot@widthfactor\wd2{%
    \hfil
    \raise\ht0\hbox{%
      \scalebox{\bigcdot@scalefactor}{%
        \lower\ht0\hbox{$#1\bullet\m@th$}%
      }%
    }%
    \hfil
  }%
}
\makeatother

\newcommand{\catname}[1]{\mathbf{#1}} 

\tcbset
{colframe=boxcolor,colback=symbols!7!pagebackground,coltext=pageforeground,
fonttitle=\bfseries,nobeforeafter,center title,size=fbox,boxsep=1.5pt,
top=0mm,bottom=0mm,boxsep=0mm,tcbox raise base}

\def\two{{\<generic>\kern0.05em\<genericb>}}
\def\twoI{{\<Ito>\kern0.05em\<Itob>}}

\def\mail#1{\burlalt{#1}{mailto:#1}}

\usepackage{thmtools} 

\declaretheorem[style=definition]{example}

\begin{document}

\title{Geometric embedding for regularity structures}
\author{Yvain Bruned$^1$, Foivos Katsetsiadis}
\institute{ 
 IECL (UMR 7502), Université de Lorraine
  \\
Email:\ \begin{minipage}[t]{\linewidth}
\mail{yvain.bruned@univ-lorraine.fr},
\\ \mail{fivosk7@gmail.com}.
\end{minipage}}

\maketitle

\begin{abstract}
In this paper, we show how one can view certain models in regularity structures as some form of geometric rough paths. This is performed by identifying the deformed Butcher-Connes-Kreimer Hopf algebra with a quotient of the shuffle Hopf algebra which is the structure underlying the definition of a geometric rough path. This provides an extension of the isomorphism between the Butcher-Connes-Kreimer Hopf algebra and the shuffle Hopf algebra. This new algebraic result relies strongly on the deformation formalism and the post-Lie structures introduced recently in the context of regularity structures.  
\end{abstract}

\setcounter{tocdepth}{1}
\tableofcontents

\section{Introduction}

\ \ \ \ \ In this work we attempt to construct a correspondence between models in regularity structures \cite{reg} and geometric rough paths from classical rough path theory \cite{Lyons98,Gubinelli2004}. Results of this kind have already been obtained in the case of branched rough paths which are another type of rough paths defined on trees (see \cite{Gub06}) instead of words. An approach given in \cite{BC} constructs a bijection between the two spaces $\BR^{\gamma}$ of branched rough paths and $\AN^{\gamma}$ of anisotropic rough paths. The main idea is to use an algebraic result from \cite{Foi02,cha10} that directly relates the underlying Hopf algebras. Inspired by this approach, we endeavour to show that certain Hopf algebras appearing in the context of regularity structures also relate to Hopf algebras with simpler presentation such as quotients of the tensor Hopf algebra $(T(\CB), \otimes, \Delta_{\shuffle})$ that appears in the context of geometric rough paths. 
Trying to find other combinatorial structures than decorated trees introduced in \cite{reg,BHZ} for SPDEs  has been developed in the recent multi-index formalism in \cite{LOT,OSSW}. The main difference with our approach, is that in their context one can hope at the best for a post-Lie morphism between decorated trees and multi-indices. Whereas, we obtain in this work an isomorphism. This duality between trees and words for coding expansions has been considered in numerical analysis (see \cite{Murua2006,Murua2017}). We also expect this work to have an impact in the context of low regularity integrators in \cite{BS} for dispersive PDEs where similar decorated trees are used.
 
Our approach further relies on an indispensable algebraic tool, which is the notion of a post-Lie algebra. In \cite{BM22}, it has been shown that certain Hopf algebras appearing in the context of regularity structures can be built directly from certain pre-Lie algebraic structures -a special case of post-Lie algebras- that are simpler to describe. This is accomplished via means of a recursive construction of the product by Guin and Oudom, first appearing in \cite{Guin1,Guin2}. This fact can reveal important information about the Hopf algebras involved. Given a pre-Lie algebra $(E, \curvearrowright)$ the Guin-Oudom procedure constructs a product on the symmetric space over the underlying vector space $E$. It also endows the space with the shuffle coproduct $\Delta_{\shuffle}$ thus turning it into a Hopf algebra, which is isomorphic to the universal enveloping algebra of the commutator Lie algebra $E_{Lie}$ associated to $E$. 

 It was already known that the graded dual of the Butcher-Connes-Kreimer Hopf algebra $\CH_{\text{\tiny{BCK}}}$ \cite{Butcher,CK1,CK2}, which is the Grossman-Larson Hopf algebra $\CH_{\text{\tiny{GL}}}$ \cite{GL}, can be generated in this manner by the free pre-Lie algebra over a set of generators which can be described as the linear span of trees endowed with the grafting product. In the work of \cite{BM22} it is proven that the graded dual of a deformed version of the Butcher-Connes-Kreimer Hopf algebra is also generated in this manner by a deformed version of the grafting product. This deformed version of the grafting product is then shown to be isomorphic to the original grafting product in the category of pre-Lie algebras via an isomorphism $\Theta$. This is illustrated below via the following diagram:

 \begin{equs}\label{diag_1}
 \begin{aligned}
\xymatrix{\curvearrowright
  \ar[rr]^{\scriptsize{\text{Guin-Oudom}}}\ar[d]_{\Theta} &&  \star  \ar[d]^{\Phi}  && \ar[ll]^{\scriptsize{\text{Dual}}}  \Delta_{\text{\tiny{BCK}}} \\ \wh \curvearrowright
\ar[rr]^{\scriptsize{\text{Guin-Oudom}}} &&  \tilde{\star} && \ar[ll]^{\scriptsize{\text{Dual}}} \Delta_{\text{\tiny{DBCK}}}
}
\end{aligned}
\end{equs}
where $  \wh \curvearrowright $ is the deformed grafting obtained from $ \curvearrowright $ by $ \Theta $. The coproducts $\Delta_{\text{\tiny{BCK}}} $ and $\Delta_{\text{\tiny{DBCK}}} $ are respectly the Butcher-Connes-Kreimer  and the deformed Butcher-Connes-Kreimer coproducts. The products $ \star $ and $ \tilde{\star} $ are the Grossman-Larson and deformed Grossman-Larson products.
The isomorphism $ \Phi $ between these two products is obtained by applying the Guin-Oudom functor to $ \Theta $ (see Theorem~\ref{theorem_functor}). 
Furthermore, the work of \cite{BK} completes this programme, in the sense that post-Lie algebras that generate the graded duals of the  Hopf algebras $ \CH_2 $ used for the recentering in singular SPDEs (see \cite{reg,BHZ,BM22}). Again, for each Hopf algebra one has an original and deformed version and these are correspondingly proven to be generated by a post-Lie product or a deformed version thereof. This could be summarise in the following diagram:
\begin{equs}\label{diag_2}
 \begin{aligned}
\xymatrix{ \wh \curvearrowright^{\text{\tiny{Lie}}}
\ar[rr]^{\scriptsize{\text{Guin-Oudom}}} &&  \star_2 && \ar[ll]^{\scriptsize{\text{Dual}}} \Delta_{2}
}
\end{aligned}
\end{equs}
We have added the notation $ \text{\tiny{Lie}} $ to stress the fact that one starts with a Lie algebra and therefore the previous deformed gafting product $ \wh \curvearrowright $ is extended to new objects. The Guin-Oudom procedure used is the one for post-Lie algebras developed in \cite{ELM}. The map $ \Delta_2 $ is the coproduct for $ \CH_2 $.

The map $ \Psi $ allows to say that the deformed Butcher-Connes-Kreimer Hopf algebra is isomorphic to the tensor Hopf algebra (see Theorem~\ref{Main}). Indeed, the basis $ B $ given by the Chapoton-Foissy isomorphism $ \Psi_{\text{\tiny{CF}}} $ is transported via $ \Phi $ in the sense that one has:
\begin{equs}
\Psi_{\text{\tiny{CF}}} : \tau_1 \star \ldots \star \tau_r \mapsto
\tau_1 \otimes \ldots \otimes \tau_r, \quad \tau_i \in B.
\end{equs} 
Then, the new isomorphism $ \Psi_{\Phi}  $ is given by 
\begin{equs}
\Psi_{\Phi} : \Phi(\tau_1) \ \tilde{\star} \ldots \tilde{\star} \ \Phi(\tau_r) \mapsto
\Phi(\tau_1) \otimes \ldots \otimes \Phi(\tau_r)
\end{equs}
where $ \Phi(\tau_1) \, \otimes \ldots \otimes \, \Phi(\tau_n) \in T(\Phi(\CB)) $ and $ \CB $ is the linear span of $ B $.  This gives a clear description of the basis that can be used in the context of the deformed Butcher-Connes-Kreimer Hopf algebra. We also know that elements of $ \CB $ are linear combinations of planted trees that are primitives elements for the Butcher-Connes-Kreimer coproduct of the form $ \mathcal{I}_{a}(\tau) $. Here, in the notation $ \tau $ is a linear combination of decorated trees and $ \mathcal{I}_{a} $ correspond of the grafting of these trees onto a new root via an edge decorated by $ a $.  

This result does cover the Hopf algebras used in \cite{BS} but not the one at play in the context of regularity structures. Indeed, not only are planted trees used for describing solutions of singular SPDEs but so are classical monomials $ X_i $. In the expansion, these objects are associated to some operators that do not commute, motivating the introduction of a natural Lie bracket.
The grafting product has to be compatible with this underlying Lie bracket and that is encapsulated in the form of a post-Lie product recently introduced for SPDEs in \cite{BK}. Therefore, the Lie- algebraic structure has to be taken into account when one extends the isomorphism introduced by Chapoton and Foissy. In our main result (see Theorem~\ref{main_result_paper}) the alphabet $ A $ is given by the $ X_i $ and the $ \Phi(\mathcal{I}_a(\tau)) \in \Phi(\CB) $. We denote by $ W $ the words on this alphabet.  The space $ \tilde{W} $ is given as the quotient of $W$ by the Hopf ideal $ \CJ $ generated by the elements
\begin{equs}
  \{ X_i \otimes \Phi( \mathcal{I}_{a}(\tau) ) - \Phi( \mathcal{I}_{a}(\tau) ) \otimes X_i - \uparrow^{i} \Phi( \mathcal{I}_{a}(\tau))  - \Phi(\mathcal{I}_{a - e_i}(\tau)) \}
\end{equs}
where $ \mathcal{I}_{a}(\tau) \in B $ and $ \uparrow^{i} $ corresponds to changing a node decoration by adding $ e_i $ to it. The $ e_i $ are the canonical basis of $ \mathbb{N}^{d+1} $. Then, there exists a Hopf algebra isomorphism $ \Psi_{\Phi}  $ between decorated trees and $ \tilde{W} $.
 The map $ \Psi_{\Phi} $ is given as an extension of $ \Psi_{\Phi} $ by
\begin{equs}
 \Psi_{\Phi} : \prod_{i=1}^n \CI_{a_i}(\tau_i) X^k \rightarrow \Psi_{\Phi}(\prod_{i=1}^n \CI_{a_i}(\tau_i)) \otimes_{j=0}^{d} \otimes_{i=1}^{k_j} X_j.
\end{equs}
where $ k = (k_0,...,k_d) \in \mathbb{N}^{d+1} $, $ X^k = \prod_{j=0}^{d} X_j^{k_j} $ and $ \prod_{i=1}^n \CI_{a_i}(\tau_i) X^k $ corresponds to a certain order as two planted trees commute but not $ X_i $ and a planted tree which is encoded in the Lie bracket. The product $ \prod_{i=1}^n $ is commutative. 

Let us comment on the main algebraic result of this paper. We know from the Milnor-Moore theorem that a Hopf algebra is the universal enveloping algebra of its primitive elements, which is defined as a quotient of a tensor algebra. The difference is that the space $ W $ appearing here is much smaller than that of the primitive elements. It is the image by an isomorphism of a basis of primitive elements for the Butcher-Connes-Kreimer Hopf algebra. The other elements are the $ X_i $. The quotient is happening with a Lie bracket between the $ X_i $ and the planted trees which is way smaller than the one taken for the Milnor-Moore theorem. The main thing to check is to see that this Lie bracket preserves the basis of Chapoton and Foissy which is the subject of  Proposition~\ref{stable_ideal}.
The construction features a very general mechanism that can be reproduced in other contexts:
 \begin{itemize}
 	\item Deformation with the help of an isomorphism that transports the structure.
 	\item Post-Lie structures: one adds new elements $ X_i $ that do not commute with the previous space. They introduce a natural Lie algebra. This new Lie algebra is used in the quotient of $ W $.
 \end{itemize}
 We obtain a better understanding of the recentering Hopf algebra that can potentially provide new results in the theory of regularity structures and the analysis of (S)PDEs.

 The paper will be structured as follows: In Section~\ref{section::2} we give an outline of the results in \cite{BC}, where the authors construct a correspondence between geometric rough paths and branched rough paths. This is accomplished using a result of Chapoton and Foissy that shows that the Grossman-Larson Hopf algebra is in fact isomorphic to the tensor Hopf algebra.
As rough paths may be seen as parametrized families of characters of these Hopf algebras one can use composition with the Hopf algebra isomorphism to directly transverse across structures see Theorem~\ref{chiso}. In this section, we also present the generalised version of the classical Butcher-Connes-Kreimer Hopf algebra $ \CH_{\text{\tiny{BCK}}} $ that can accomodate decorations on the edges and vertices of the forests. We also present the generalized version of the Grossman-Larson Hopf algebra $ \CH_{\text{\tiny{GL}}}  $ which again comes with decorations on the edges and vertices and is dual to $ \CH_{\text{\tiny{BCK}}} $. We also present the Chapoton-Foissy isomorphism) in this context (see Theorem~\ref{Foissy_Chapoton_basis}) that we will use in the sequel. 
 
  In Section~\ref{section::3}, we present the grafting and deformed grafting pre-Lie products and explain how they generate the Grossman-Larson product (see Theorem~\ref{Free pre-Lie -> Grossman - Larson}) and a deformed version thereof (see Theorem~\ref{Deformed Grafting0}), via a contruction given by Guin and Oudom. We then present the isomorphism $\Theta$ appearing in \cite{BM22} between the grafting and deformed grafting pre-Lie algebras. We use this to prove that this deformed structure is isomorphic as a Hopf Algebra to the original Grossman-Larson Hopf algebra. Hence, by virtue of the result of Chapoton and Foissy, it is also isomorphic to the tensor Hopf algebra (see Theorem~\ref{Main}).

 In Section~\ref{section::4} we present the post-Lie algebraic formalism and the generalisation of the Guin-Oudom construction in that context. We recall the results in \cite{BK} and make essential use of them to prove our main result. The post-Lie algebraic formalism allows for a precise encoding of the action of the deformed grafting product alongside its interaction with the $\uparrow^i$ operators. Uploading this data onto the universal enveloping algebra via the construction in \cite{ELM} allows for a finer analysis of the $\CH_{2}$ Hopf algebra and the relation it bears to the simpler deformed Grossman-Larson Hopf algebra. This information, together with the isomorphism $\Psi_{\Phi}$ ultimately leads to the main result Theorem~\ref{main_result_paper} which is an isomorphism between the $\CH_{2}$ Hopf algebra and an appropriate quotient of the shuffle Hopf algebra.

 Finally in Section~\ref{section::5}, we explore some applications in the context of regularity structures. We show how, using this isomorphism one may move from one encoding to another. This is done by composing the elements of the structure group with the isomorphism and obtaining a new structure group acting on the relevant quotient of the tensor Hopf algebra. This is similar in spirit to the approach in \cite{BC}. The main algebraic result of the section is Theorem~\ref{T_+ isomorphism} which identifies $ \CT_+ $ the vector space used for the structure group as isomorphic to a Hopf subalgebra of the quotient of the shuffle Hopf algebra described before. Then, we propose an attempt to rewrite Theorem~\ref{chiso} int the context of regularity structures (see Theorem~\ref{chiso_bis}).
 
 \subsection*{Acknowledgements}
 
 {\small
 
 	Y. B. gratefully acknowledges funding support from the European Research Council (ERC) through the ERC Starting Grant Low Regularity Dynamics via Decorated Trees (LoRDeT), grant agreement No.\ 101075208.
 	Views and opinions expressed are however those of the author(s) only and do not necessarily reflect those of the European Union or the European Research Council. Neither the European Union nor the granting authority can be held responsible for them.
 	F. K. was funded by the School od Mathematics at the University of Edinburgh during the writing of this work.
 }

\section{From non-geometric to geometric rough paths}

\label{section::2}

In this section, we present the state of the art in the context of moving from geometric to branched rough paths. We formulate the results using the notion of anisotropic and branched $\gamma$ rough paths, with the corresponding spaces denoted by $\mathbf{ARP}^{\gamma}$ and $\mathbf{BRP}^{\gamma}$. We shall give the definitions of the relevant Hopf algebras and then proceed to outline the results in \cite{HK15}, \cite{BC} and \cite{TZ}. Our approach in the next sections is a generalization of the approach presented in \cite{BC}.

\ \ \ \ Before moving further, we introduce some notations. We shall denote by $T_{E}^{V}$ the set of all rooted trees with vertices decorated by $V$ and edges decorated by $E$ and by $F_{E}^{V}$ the set of forests which consists of monomials over $T_{E}^{V}$. We then denote by $\CT_{E}^{V}$ the formal linear span of $T_{E}^{V}$. We also denote by $\CF_{E}^{V}$ the forest algebra, which consists of all polynomials over $T_{E}^{V}$. It is the free commutative algebra over the vector space $\CT_{E}^{V}$. 

 Furthermore, we denote by $ P_E^V $ the set of planted trees and by $ \mathcal{P}_E^V $ their linear span. \label{planted_p} A planted tree is of the form $ \mathcal{I}_{a}(\tau) $ where $ \tau \in \mathcal{T}_E^V $ and $ \mathcal{I}_{a}(\tau) $ denotes the grafting of the tree $ \tau $ onto a new root with no decoration via an edge decorated by $a \in E$. We also use $N_{\tau}$, $E_{\tau}$ and $L_{\tau}$ for the set of vertices, edges and leaves of a tree $\tau \in T_{E}^{V}$.
We may equip these structures with different products and at these times we will essentially use the notation to refer to the underlying vector spaces. 

 We will use $\catname{Vec}$ for the category of vector spaces and $\catname{Alg}$ and $\catname{CAlg}$ for the category of algebras and commutative algebras respectively. We shall use $S: \catname{Vec} \rightarrow \catname{CAlg}$ for the symmetric algebra functor taking a vector space $V$ to the free commutative algebra over $V$. Similarly, we use $T: \catname{Vec} \rightarrow \catname{Alg}$ for the tensor algebra functor taking a vector space $V$ to the free associative algebra over $V$. 

 We also define an admissible cut of a tree to be any selection of edges that are pairwise incomparable with respect to the partial ordering induced by the tree. If $h \in \CF_{E}^{V}$ then we use $\text{Adm}(h)$ to denote the set of admissible cuts of $h$. 

\begin{definition}
We define the Butcher-Connes-Kreimer coproduct $\Delta_{\tiny{\text{BCK}}}$ on the symmetric algebra $S(\mathcal{P}_{E}^{V})$ by setting, for $h \in \mathcal{P}_{E}^{V}$:
\begin{equs} \label{Connes_Kreimer_1}
\Delta_{\tiny{\text{BCK}}} (h) := \sum_{C \in \scriptsize{\text{Adm}}(h) }  R^C(h) \otimes \tilde{P}^C(h)
\end{equs} 
 Here, we have used $\tilde{P}^C(h)$ to denote the pruned forest that is formed by collecting all the edges at or above the cut, including the ones upon which the cut was performed, so that the edges that were attached to the same node in $h$ are part of the same tree. The term $R^C(h)$  corresponds to the "trunk", that is the subforest formed by the edges not lying above the ones upon which the cut was performed.   In the case of decorated trees with no decorations on the edges, we consider the classical Butcher-Connes-Kreimer coproduct given by:
 \begin{equs} \label{Connes_Kreimer_2}
\hat{\Delta}_{\tiny{\text{BCK}}} (h) := \sum_{C \in \scriptsize{\text{Adm}}(h) }  R^C(h) \otimes P^C(h)
\end{equs} 
where this time, we do not keep the edges in the cut $ C $ with $ P^{C}(h) $.
\end{definition}

The Butcher-Connes-Kreimer Hopf algebra $\CH_{\text{\tiny{BCK}}}$ is the graded bialgebra with underlying algebraic structure given by the natural symmetric product on $ S(\mathcal{P}^V_E) $ and coalgebra structure given by $ \Delta_{\text{\tiny{BCK}}}  $. The grading is defined to be the number of edges. As a graded connected bialgebra, it is also a Hopf algebra. We denote the usual Butcher-Connes-Kreimer Hopf algebra by $\hat{\CH}_{\text{\tiny{BCK}}}$ given by $ S(\mathcal{T}^V) $ equipped with the forest product and coproduct $ \hat{\Delta}_{\text{\tiny{BCK}}}  $. The space $ \mathcal{T}^V $ si the linear span of decorated trees with only deorations on the vertices. For this Hopf algebra, the grading is the number of nodes.

\begin{remark} The coproduct $ \hat{\Delta}_{\text{\tiny{BCK}}} $ is used for branched rough paths. In the context of SPDEs with several equations, one has to keep track of the various operators needed for a given iterated integral. This is perfomed by decorations on the edges. Let us mention, that a variant of the Butcher-Connes-Kreimer coproduct has been used in the context of Volterra-type rough paths in \cite{BK1} where the edges cut are also kept for keeping track the fact that they were attached to the same node. This is crucial for proving a generalized Chen's relation involving a convolution-type product.
\end{remark}

\begin{remark}
We shall frequently use Sweedler's notation and will write $\Delta_{\text{\tiny{BCK}}}(h) = \sum_{(h)} h^{(1)} \otimes h^{(2)}$ to denote the sum ranging over the expansion of the coproduct $\Delta_{\text{\tiny{BCK}}}$. We will also frequently use Sweedler's notation for the other coproducts that will be introduced.
\end{remark}

One can provide a recursive formula for the  Butcher-Connes-Kreimer coproduct  $ \Delta_{\text{\tiny{BCK}}}$:
\begin{equs} \label{recur_def_deltaCK}
\begin{aligned}
\Delta_{\text{\tiny{BCK}}} \one &  = \one \otimes \one \\
\Delta_{\text{\tiny{BCK}}}  
\CI_a(X^k \tau) & = 
\left( \id  \otimes \CI_a(X^k \cdot) \right) \Delta_{\text{\tiny{BCK}}} \tau + \CI_a(X^k \tau) \otimes \one,
\end{aligned}
\end{equs}
and it extends multiplicatively to the product of $ S(\mathcal{P}^V_E) $.  Here, $ \tau  $ belongs to $ S(\mathcal{P}^V_E) $. The notation $ X^k \tau $ with $ k \in V $ means that we identify the roots of the planted trees in $\tau$ into a single root decorated by $k$. The map $ \CI_a(X^k \cdot) $ is an operator from $ S(\mathcal{P}^V_E) $ into $ S(\mathcal{P}^V_E) $ that grafts the decorated tree $ X^k \tau $ onto a new root via an edge decorated by $ a $.
From this coproduct, one has an associative product denoted as the Grossman-Larson product $ \star $ defined as:
\begin{equs}
\sigma \star \tau := \left( \sigma \otimes \tau \right) \Delta_{\text{\tiny{BCK}}}
\end{equs}
where we use the following identification by viewing  $ \tau \in S(\mathcal{P}^V_E) $ as a linear functional in  $  S(\mathcal{P}^V_E) $ such that : $ \langle\tau,\bar \tau \rangle =  S(\tau) $ if $ \tau = \bar \tau $ and zero elsewhere. The coefficient $ S(\tau) $ corresponds to the internal symmetry factor of the forest $ \tau $. It is given by $ \prod_{i} |\text{Aut}(\tau_i)| $ where the $ \tau_i  $ are the trees of $ \tau $ and $  |\text{Aut}(\tau_i)| $ is the number of automorphisms preserving $ \tau_i $.
 We define a second coproduct on $ S(\mathcal{P}^V_E) $ as
\begin{equs}
\Delta \tau = \tau \otimes \one + \one \otimes \tau
\end{equs}
for every $ \tau \in \mathcal{P}^V_E $
and then extends multiplicatively for the symmetric product of $ S(\mathcal{P}^V_E) $.

For the rest of the section, we set  $V = \{0,...,d\}$. We denote by $ \CT$ the set of  decorated trees whose nodes are decorated by $ V $ and edges are not decorated. The set $  \CT_N $ consists of decorated trees of $ \CT $  with at most $ N  $ nodes.
We also set $ \CF $ to be forests formed of decorated trees in $\CT$ and $ \CF_N $ forests with at most $ N $ nodes. 
We denote by $ \CG $ (resp. $ \CG_N $) the set of characters from the Butcher-Connes-Kreimer Hopf algebra $\hat{\CH} $ (resp. $\hat{\CH}_N = \langle \CF_N \rangle $) into $\mathbb{R}$. These are linear algebra morphisms forming a group with respect to the convolution product $ \star_0 $ 
\begin{equs}
\label{convolution}
	X \star_0 Y :=  (X \otimes Y ) \hat{\Delta}_{\text{\tiny{BCK}}}. 
\end{equs}
The unit for the convolution product is the co-unit $\mathbf{1}^*$ which is non-zero only on the empty tree.

\begin{definition}

Let $\gamma\in\,]0,1[$, we define a branched $\gamma$-rough path as a map $X:[0,1]^2\to\cal G$ such that $X_{tt}=\one^{*}$ and such that Chen's relation is satisfied:
  \begin{equs} \label{chen}
  X_{su} \star_0 X_{ut}=X_{st}, \qquad s,u,t\in[0,1],
  \end{equs}
  and the analytical condition
  \begin{equs} \label{analytical_0}
  \sup_{0 \leq s,t \leq 1} \frac{\langle  X_{st}, \tau \rangle}{|t-s|^{(1-\gamma)|\tau|_0 + \gamma |\tau|}} < \infty,
  \end{equs}
  where $ |\tau|_0 $ counts the number of times the decoration $ 0 $ appears in $ \tau $. In the sequel, we will consider the biggest $ N \in \mathbb{N} $ such that  $ \gamma N \leq 1 $. The branched $ \gamma $-rough paths are taking values in $ \CG_{N} $.
 We denote this space by $\BR^\gamma$.
\end{definition}

 We also introduce the shuffle Hopf algebra. Given an alphabet $ A $, we consider the linear span of the words on this alphabet denoted by $ T(A) $. We set $ \varepsilon  $ as the empty word. The product on $ T(A)$ is the shuffle product defined by
  \begin{equs}
  \varepsilon \shuffle v=v\shuffle \varepsilon  =v, \quad (au\shuffle bv) = a(u\shuffle bv) + b(au\shuffle v)
  \end{equs}
for all $u,v\in T(A)$ and $a,b\in A$.
We first define the deshuffle coproduct $ \Delta_{\shuffle} $ dual to the shuffle product defined for every $ a \in A $ as
\begin{equs}
\Delta_{\shuffle} a = a \otimes \varepsilon + \varepsilon \otimes a
\end{equs}
and then extends multiplicatively for the tensor product $ \otimes $. The coproduct $\bar\Delta:T(A)\to T(A)\otimes T(A)$ is the deconcatenation of words: 
\[ \bar\Delta(a_1\dotsm a_n) = a_1\dotsm a_n\otimes \varepsilon  + \varepsilon \otimes a_1\dotsm a_n + \sum_{k=1}^{n-1}a_1\dotsm a_k\otimes a_{k+1}\dotsm a_n.\]
Equipped with the shuffle product and the deconcatanation coproduct, $ T(A) $  is a Hopf algebra. The grading of $ T(A) $
 is given by the length of words $\ell(a_1\dotsm a_n) = n$. We denote by $ \CG_A $ the group of characters associated to $ T(A) $ and by $ * $ the convolution product.

\begin{definition} 
 
  An anisotropic $\gamma$-rough path, with $\gamma=(\gamma_a, \, a\in A)$, $ 0<\gamma_a<1$, is a map $X:[0,1]^2\to \CG_{A}$ such that $ X_{tt} = \varepsilon^{*} $ where $  \varepsilon^{*}$ is the counit. It satisfies
\begin{equs}
X_{su} * X_{ut}=X_{st}, \qquad |\langle X_{st},v\rangle|\lesssim|t-s|^{\hat{\gamma}\omega(v)}
\end{equs} 
for all $(s,u,t)\in[0,1]^3$ and words $v$. Here, $\hat{\gamma}=\min_{a\in A}\gamma_a$ and for a word $v=a_1\dotsm a_k$ of length $k$ we define
\begin{equation}
  \label{eq:weight}
  \omega(v)=\frac{\gamma_{a_1}+\dotsc+\gamma_{a_k}}{\hat{\gamma}}=\frac{1}{\hat{\gamma}}\sum_{a\in A} n_a(v)\gamma_a
\end{equation}
where $n_a(v)$ is the number of times the letter $ a $ appears in $ v $.
The space of anisotropic $\gamma$-rough paths is denoted by $ \AN^{\gamma} $. When the $ \gamma_a $ are all equal to a fixed $\gamma $, one recovers the classical geometric rough paths.
\end{definition}

 As for the branched rough paths, we perform a truncation and consider paths taking values in $ \CG_{\CT_N,N}$. Elements of $ \CG_{\CT_N,N}$  are characters over $ T_N(\CT_N) $ which are  words $ v = \tau_1 \cdots \tau_n $ built on the alphabet $ \CT_N $ such that $  \sum_{i=1}^n |\tau_i| \leq N $. 
 
 


 The first approach to moving from trees into words is given by the Hairer-Kelly map $ \Psi_{\text{\tiny{HK}}}  $ in the context of geometric rough paths in \cite{HK15}. 
  This map first introduced in \cite{HK15}  is given in \cite[Def. 4, Sec. 6]{Br173} as the unique Hopf algebra morphism from $\mathcal{H}$ to the shuffle Hopf algebra $(T(\mathcal{T}), \shuffle,\bar\Delta)$ obeying: 
\[	\Psi_{\text{\tiny{HK}}} = (\Psi_{\text{\tiny{HK}}} \otimes P_{\one}) \hat{\Delta}_{\text{\tiny{BCK}}} 
\]
where  $P_{\one}:=\mathrm{id}-\one^{*}$ is the augmentation projector.
The following theorem given in \cite{TZ} established a correspondence between  anisotropic rough paths and branched rough paths:
\begin{theorem}\label{Hairer-Kelly}
  Let $X$ be a branched $\gamma$-rough path. There exists an anisotropic geometric rough path $\bar{X}$ indexed by words on the alphabet $\CT_N$, $ N = \lfloor 1/\gamma \rfloor $, with exponents $(\gamma_\tau, \tau\in\CT_N)$, and such that $\langle X,\tau\rangle = \langle\bar{X},\Psi_{\text{\tiny{HK}}}(\tau)\rangle$.
\end{theorem}

\begin{remark}
The previous theorem relies on the  Lyons-Victoir extension theorem given in \cite{LV07}  which is not canonical. 
The authors in \cite{TZ} identified   a transitive free action of the additive group $\CC^\gamma$ on $ \BR^\gamma $. The abelian group $\CC^\gamma$ is given by  
\[
  \cal C^\gamma:=\{(g^\tau)_{\tau\in \CT_N}: \, g^\tau_0=0, \, g^\tau\in C^{\gamma_\tau}([0,1]), \, \forall\, \tau\in\CT, |\tau|\leq N\}.
\]
Explicit expressions for $ g $ have been given in \cite{Br20} for the BPHZ renormalisation at the level of rough paths introduced in \cite{BCFP}.
Parametrisation in the context of regularity structures has been considered in \cite{BB21b}.
\end{remark}

 Lastly, the approach most relevant to this work, given in \cite{BC}, constructs a bijection between the two spaces 
$ \BR^{\gamma} $ and $ \AN^{\gamma} $. The main idea is to use an algebraic result from \cite{Foi02,cha10}. We denote by $ \hat{\mathcal{H}}_{\text{\tiny{GL}}} $ the Grossmann-Larson Hopf algebras (resp. $ \mathcal{H}_{\text{\tiny{GL}}} $)  defined on $ S(\mathcal{T}) $ (resp. $ S(\mathcal{P}^V_E) $ ) equipped with the product $ \star_0 $ (resp. $ \star $) and the coproduct $ \Delta $. We recall this result of Chapoton and Foissy \cite{Foi02,cha10}:
\begin{theorem} \label{Foissy_Chapoton_basis}
There exists a subspace $ \hat{\CB} = \langle\tau_1, \tau_2,...\rangle $ of $ \CT $ (resp. $ \mathcal{B} $ and  $ \mathcal{P}^V_E $) such that $ \hat{\CH}_{\text{\tiny{GL}}} $ (resp. $ \CH_{\text{\tiny{GL}}} $) is isomorphic as a Hopf algebra to the tensor Hopf algebra $ (T(\hat{\CB}), \otimes, \Delta_{\shuffle})$ (resp. $ (T(\CB), \otimes, \Delta_{\shuffle})$) which consists of the linear span of the set of words from the alphabet $\hat{\CB}$ (resp. $ \CB $), endowed with the tensor product and the shuffle coproduct. 
\end{theorem}

We provide an outline here of the construction. First, one proves the existence of a set 
\begin{equs} \label{basis_B}
B = \{ \tau_{1}, \tau_{2}, ... \}
\end{equs}
 that consists of a basis of primitive elements of the Hopf algebra $\CH_{\text{\tiny{BCK}}}$ belonging to $ \mathcal{P}^V_E $ such that every $\tau \in \CH_{\text{\tiny{GL}}}$ has a unique representation of the form:
\begin{equs} \label{uniquedec}
\tau = \sum_{R} \lambda_R \tau_{r_1} \star \ldots \star \tau_{r_n}
\end{equs}
where the sum is performed over finitely many multi-indices $R = (r_1,\ldots,r_n)$. 
Then, one can exhibit an isomorphism $ \Psi_{\text{\tiny{CF}}} $ between the two Hopf algebras  $ \CH_{\text{\tiny{GL}}} $ and $ T(\CB) $ where $ \CB$ is the linear span of $B$ as follows:
\begin{equs}
\Psi_{\text{\tiny{CF}}} : \tau_1 \star \ldots \star \tau_r \mapsto
\tau_1 \otimes \ldots \otimes \tau_r
\end{equs}
where $ \tau_1 \otimes \ldots \otimes \tau_r \in T(\CB) $. This will be the isomorphism that we will use in the next section. We will obtain an isomorphism of the deformed Grossman-Larson Hopf algebra $\CH_{\text{\tiny{DGL}}}$ with the tensor Hopf algebra  $ (T(\CB), \otimes, \Delta_{\shuffle})$.

In the context of rough paths, one uses the  isomorphism $ \hat{\Psi}_{\text{\tiny{CF}}} $ between the two spaces  $ \CH_{N}^{*} $ and $ T_N(\hat{\CB}_{N}) $ based on the basis $ \hat{\CB}_N $ (see \cite[Lemma 4.2]{BC}):
\begin{equs}
\hat{\Psi}_{\text{\tiny{CF}}} : \tau_1 \star_0 \ldots \star_0 \tau_r \mapsto
\tau_1 \otimes \ldots \otimes \tau_r
\end{equs}
where $ \tau_1 \otimes \ldots \otimes \tau_n \in T_N(\hat{\CB}_{N}) $. Here $ \hat{\CB}_{N} $ are elements of $ \hat{\CB} $ with at most $ N $ nodes.
One has from \cite{BC}

\begin{theorem} \label{chiso} Let $ X \in \BR^{\gamma} $, then $ \hat{X} := \hat{\Psi}_{\text{\tiny{CF}}}(X) \in \AN^{\gamma}$.
\end{theorem}

In \cite{Br20}, the action of the renormalisation on this construction has been described. The family of renomalisation maps considered are BPHZ renormalisation map $ M $ (inspired from the BPHZ renormalisation of Feynman diagrams \cite{WZ69,Hepp,BP} which was used in the context of regularity structures \cite{BHZ,CH16}) whose adjoints $ M^{*} $ are morphisms for the product $ \star_0 $:
\begin{equs}
M^{*} \left( \tau \star_0 \sigma \right) = M^{*} \tau \star_0 M^{*} \sigma. 
\end{equs}
Then, one is able to define a renormalisation map $ \hat{M}^{*} $ on $ T_N(\hat{\CB}_{N}) $ that commutes with the isomorphism $ \hat{\Psi}_{\text{\tiny{CF}}} $ (see \cite[Theorem 4.7]{Br20}):
\begin{equs}
\hat{M}^{*} \hat{X} = \hat{M}^{*} \hat{\Psi}_{\text{\tiny{CF}}}(X)  =   \hat{\Psi}_{\text{\tiny{CF}}}( M^{*}  X).
\end{equs} 
BPHZ renormalisation maps in the context of rough paths have been first considered in \cite{BCFP} with some examples provided in \cite{BCF}.


\section{An isomorphism for the deformed Grossman-Larson Hopf algebra}

\label{section::3}

In this section, we introduced pre-Lie and multi-pre-Lie algebras with the main example being the grafting product for decorated trees and its deformations given in \cite{BCCH,BM22}. Then, we apply the Guin-Oudom procedure \cite{Guin1, Guin2} to the grafting product for deriving the Grossman-Larson Hopf algebra $\CH_{\text{\tiny{GL}}}$, the graded dual of the Butcher-Connes-Kreimer Hopf algebra $\CH_{\text{\tiny{BCK}}}$. Using the functor given by Guin-Oudom, we are able to lift the isomorphism $ \Theta $ introduced in \cite{BM22} (see \eqref{recursive_theta}) at the level of the grafting products to an isomorphism $ \Phi $ for the deformed Grossman-Larson Hopf algebras (see Theorem~\ref{theorem_functor}). 
This allows us to state our main result which is Theorem~\ref{Main} that translates the Chapoton-Foissy isomorphism in the context of the deformed Grossman-Larson Hopf algebra: One just applies the isomorphism $ \Phi $ to the basis previously obtained in Theorem~\ref{Foissy_Chapoton_basis}.

\ We begin this section by giving the definition of a pre-Lie algebra.

\begin{definition}
A pre-Lie algebra is an algebra $(\mathcal{P}, \curvearrowright)$ over a field $\mathbf{k}$ of characteristic $0$, whose product satisfies the following relation for every $ x,y,z \in \mathcal{P} $ 
$$
x \curvearrowright (y \curvearrowright z) - (x \curvearrowright y) \curvearrowright z = y \curvearrowright (x \curvearrowright z) - (y \curvearrowright x) \curvearrowright z.
$$
\end{definition}

\begin{remark}
Note that every associative algebra is a pre-Lie algebra as in this case the associator vanishes and the left- and right-hand sides above are both equal to zero. 
\end{remark}

A pre-Lie algebra gives rise to a Lie algebra:

\begin{proposition}
If $(E, \curvearrowright)$ is a pre-Lie algebra, then the commutator $[x, y] = x \curvearrowright y - y \curvearrowright x$ is a Lie bracket.
\end{proposition}

\begin{remark}
Here is an equivalent definition: An algebra $(E, \curvearrowright)$ over a field $\mathbf{k}$ of characteristic $0$, whose commutator is a Lie bracket and left multiplication by $\curvearrowright$ gives a representation of the commutator Lie algebra.
\end{remark}

\begin{remark}
Not every Lie algebra comes from an associative algebra (with the commutator as its Lie bracket). For example, free Lie algebras do not arise from any associative algebra. They do, however, arise from free pre-Lie algebras, see \cite{cha10}. It is interesting to study the implications of a Lie algebra $L$ arising from a pre-Lie structure. An explicit recursive procedure, given by Guin and Oudom (see Theorem~\ref{Guin-Oudom}), for constructing an associative product on the symmetric space over $L$ is one implication. The free cocommutative coalgebra endowed with this associative product turns out to be a Hopf algebra that is isomorphic to the universal enveloping algebra of $L$. This can be seen as exploiting the pre-Lie structure to obtain extra information about the universal envelope of $L$.
\end{remark}

 We also give the definition of a multi-pre-Lie algebra first introduced in \cite{BCCH}. Although a seemingly richer structure, all the information can be condensed into a single pre-Lie algebra. It is nonetheless a useful notion when describing certain families of products.

\begin{definition}
A multi-pre-Lie algebra indexed by a set $E$ is a vector space $\mathcal{P}$ over a field $\mathbf{k}$ of characteristic $0$, endowed with a family $(\curvearrowright^{\alpha})_{\alpha \in E}$ of bilinear products such that for every $ x, y,z \in \mathcal{P} $
$$
x \curvearrowright^{a} (y \curvearrowright^{b} z) - (x \curvearrowright^{a} y) \curvearrowright^{b} z = y \curvearrowright^{b} (x \curvearrowright^{a} z) - (y \curvearrowright^{b} x) \curvearrowright^{a} z.
$$
As it is shown below (see \cite{Foi21}), one can summarise all the data of a multi-pre-Lie algebra into a single pre-Lie algebra.
\end{definition}
\begin{lemma}
If $\CP$ is a multi-pre-Lie algebra over a field $\mathbf{k}$ of characteristic $0$ and indexed by a set $E$, then $\CP \otimes \mathbf{k}E$ is a pre-Lie algebra when endowed with the product
\begin{equs}
(x \otimes a) \curvearrowright (y \otimes b) = (x \curvearrowright^{a} y) \otimes b
\end{equs}
for any $a, b \in E$ and for any $x,y,z \in \mathcal{P}$.
\end{lemma}

\begin{example}\rm
 \label{rooted_r} A family of pre-Lie products on $ \mathcal{T}^V_E $ is given by grafting by means of decorated edges, namely:
\begin{equation}
	\label{grafting_a}
\sigma \curvearrowright^a \tau:=\sum_{v\in  N_{\tau} } \sigma \curvearrowright^a_v  \tau,
\end{equation}
where $\sigma $ and $\tau$ are two decorated rooted trees and where $\sigma \curvearrowright^a_v \tau$ is obtained by grafting the tree $\sigma$ on the tree $\tau$ at vertex $v$ by means of a new edge decorated by $a\in E$. 
\end{example}

Another example is a deformed version of the above family of grafting products.

\begin{example}

We suppose here that the vertices are decorated by elements of a monoid $\Omega$ and that $\Omega = \mathbb{N}^{d+1}$ here, endowed with componentwise addition. A grading is given by 
\begin{equs}
|\mathbf n|_{\s}:=s_0n_0+\cdots +s_{d}n_{d}
\end{equs}
 where $\mathfrak s:=(s_0,\ldots, s_{d})\in \mathbb{N}_{>0}^{d+1}$ is fixed.   We suppose that $V = S \times \mathbb{N}^{d+1}$ and $E = S' \times \mathbb{N}^{d+1}$  where $ S $ and $ S' $ are two finite sets. Then $\Omega$ acts freely on both $E$ and $V$ in a graded way. We denote by $+$ the addition in $\Omega$ as well as both actions of $\Omega$ on $E$ and $V$. A family of deformed grafting products on $\mathcal T_E^V$ is defined as follows:
\begin{equation} \label{deformation_preLie}
\sigma \widehat{\curvearrowright}^a \tau:=\sum_{v\in N_{\tau}}\sum_{\ell\in\N^{d+1}}{\Labn_v \choose \ell} \sigma  \curvearrowright_v^{a-\ell}(\uparrow_v^{-\ell} \tau).
\end{equation}
\label{deformed_grafting_a}
Here $ \Labn_v \in\mathbb{N}^{d+1}$ denotes the second component of the decoration at the vertex $ v $. The generic term is self-explanatory if there exists a (unique) pair $(b,\alpha)\in E\times V$ such that $a=\ell+b$ and $\Labn_v =\ell+\alpha$. It vanishes by convention if this condition is not satisfied. The operators $\uparrow_{v}^{\omega}$ act by adding $\omega$ to the decoration $\Labn_{v}$. We define the \textsl{grading} of a tree in $\mathcal T_E^V$  by the sum of the gradings of its edges given by $ |\cdot|_{\text{grad}} $:
\begin{equation} \label{grading}
|\tau|_{\text{grad}}:=\sum_{e\in E_{\tau}}\big|\Labe(e) \big|_{\s}.
\end{equation}
where $ \Labe(e) $ is the decoration of the edge $ e $.
Then, $ \widehat{\curvearrowright}^a $ is a deformation of $ \curvearrowright^a $ in the sense that:
\begin{equs}
\sigma \widehat{\curvearrowright}^a \tau = \sigma \curvearrowright^a \tau +
\hbox{ lower grading terms}.
\end{equs}
\end{example}

We now proceed to give the definitions of pre-Lie products that will be of interest to us. We first define the grafting product, a pre-Lie product on planted trees with edge and vertex decorations that gives rise to the free multi-pre-Lie algebra equipped with a family of pre-Lie products over a prescribed set of generators. We will then define a deformed version of this grafting product that is of greater interest to us, which turns out to be isomorphic to the original grafting product.

\begin{definition}[Grafting product]
The grafting product on the space $\mathcal{P}_{E}^{V}$ is defined as follows for planted trees and is then extended by linearity:  
\begin{equs} \label{def_pre_a}
\mathcal{I}_{a}(\sigma) \curvearrowright \mathcal{I}_{b}( \tau) =   \mathcal{I}_{b}( \sigma \curvearrowright^{a}  \tau)
\end{equs}
for any $a,b \in E$ and any $\tau, \sigma \in \mathcal{T}_{E}^{V}$.
\end{definition}

\begin{definition}[Deformed grafting product]
The deformed grafting product on the space $\mathcal{P}_{E}^{V}$ is defined as follows for planted trees and is then extended by linearity:  
\begin{equation}
\quad \mathcal{I}_{a}(\sigma) \, \wh\curvearrowright \, \mathcal{I}_{b}( \tau) =   \mathcal{I}_{b}( \sigma \wh\curvearrowright^{a}   \tau)
\end{equation}
for any $a,b \in E$ and any $\tau, \sigma \in \mathcal{T}_{E}^{V}$.
\end{definition}

\begin{remark}

The grafting and deformed grafting products are clearly pre-Lie products under the identification of the space $\CP_{E}^{V}$ with $\CT_{E}^{V} \otimes \mathbf{k}E$ by identifying an element $\CI_{a}(\tau)$ with $\tau \otimes a$.

\end{remark}

We will now introduce a deformed version of the original Butcher-Connes-Kreimer coproduct, which we call the deformed Butcher-Connes-Kreimer (DBCK) coproduct and denote by $\Delta_{\text{\tiny{DBCK}}}$.

\begin{definition}
We suppose that the set $V$ of vertex decorations coincides with the commutative monoid $\Omega$. Then, the deformed Butcher-Connes-Kreimer coproduct is defined by the maps $ \Delta_{\text{\tiny{DBCK}}} : S(\mathcal{P}_E^V ) \rightarrow S(\mathcal{P}_E^V ) \otimes S(\mathcal{P}_E^V )$ and $ \bar{\Delta}_{\text{\tiny{DBCK}}} : \mathcal{T}_E^V  \rightarrow S(\mathcal{P}_E^V ) \otimes \mathcal{T}_{E}^V   $ \label{delta_dck} \label{delta_dck_bar} defined recursively by:
\begin{equs} \label{recur_def_deltaDCK}
\begin{aligned}
\Delta_{\text{\tiny{DBCK}}}  \CI_a(\tau) & = 
\left( \id  \otimes \CI_a \right) \bar{\Delta}_{\text{\tiny{DBCK}}} \tau + \CI_a(\tau) \otimes \one,
\quad \bar{\Delta}_{\text{\tiny{DBCK}}} X^k  = \one \otimes X^{k} 
\\ \bar{\Delta}_{\text{\tiny{DBCK}}} \CI_a(\tau) & = \left( \id \otimes \CI_a \right)  \bar{\Delta}_{\text{\tiny{DBCK}}} \tau + \sum_{\ell \in \N^{d+1}} \frac{1}{\ell !} \CI_{a + \ell}(\tau) \otimes X^{\ell}.
\end{aligned}
\end{equs}
The map $ \Delta_{\text{\tiny{DBCK}}}  $ is extended using the product of $ S(\mathcal{T}_E^V) $. We use the tree product for extending the map $ \bar{\Delta}_{\text{\tiny{DCK}}} $. Here the tree product is the merging root product. It means that given two decorated trees, their tree product is equal to a new decorated tree obtained by identifying the roots of the two trees and adding decorations from the previous roots to the new root.
The infinite sum over $ \ell $ makes sense via a bigrading introduced in \cite[Section 2.3]{BHZ}.
\end{definition}

\begin{definition}
The Deformed Butcher-Connes-Kreimer (DBCK) Hopf algebra $\CH_{\text{\tiny{DBCK}}}$ is the graded bialgebra on $\CF_{E}^{V} = S(\mathcal{P}_E^V )$ equipped with the forest product (i.e. the product of the symmetric algebra) and the $\Delta_{\text{\tiny{DBCK}}}$ coproduct. As a graded, connected bialgebra, it is also a Hopf algebra.
\end{definition}

For a Hopf algebra, we shall denote the space of primitive elements of $\CH$ by $Prim(\CH)$. Note that $Prim(\CH)$ is a linear subspace of $\CH$. Equipping $Prim(\CH)$ with the commutator Lie bracket $[h_{1}, h_{2}] = h_{1}h_{2} - h_{2}h_{1}$ it is also a Lie algebra.
It is well-known that the primitive elements of a Hopf Algebra carry the structure of a Lie algebra. When $\CH$ is a cocommutative graded connected Hopf algebra with finite-dimensional graded components, the Milnor-Moore theorem tells us that $\CH \cong U(Prim(\CH))$, i.e. that $\CH$ is isomorphic as a Hopf algebra to the universal enveloping algebra over it's primitives. When $\CH$ is cofree-cocommutatve and right-sided, the primitive elements of $\CH$ admit a finer structure, that of a pre-Lie algebra. An explicit description of any Hopf algebra obeying these conditions via means of it's underlying pre-Lie algebra is given by the Guin-Oudom procedure \cite{Guin1,Guin2}, which gives a recursive construction of the algebra's associative product on the symmetric algebra $S(Prim(\CH))$ over the primitives:

\begin{theorem}[Guin-Oudom]
\label{Guin-Oudom}
Let $(\mathcal{P}, \triangleright )$ be a pre-Lie algebra and let $S(\mathcal{P})$ denote the symmetric space over the underlying vector space.
For every  $u,v, w \in S(\mathcal{P})$, $x, y  \in \mathcal{P}$ we start by extending the product $\triangleright$ on $S(\mathcal{P})$ as follows:
\begin{equs}\label{e::symmetric product}
\begin{aligned}
\one \triangleright w&= w, \quad 
u \triangleright \one =\one^{\star}(u),\\
w \triangleright uv &=\sum_{(w)} ( w^{(1)} \triangleright u)(w^{(2)} \triangleright  v),\\
 x v \triangleright y  &=  x \, \triangleright \,   ( v \triangleright y)  - (x \, \triangleright \, v) \triangleright y
 \end{aligned}
\end{equs}
 where $\one^\star$ stands for the counit, the summation over $(w)$ is shorthand for summing over the terms of the expansion for the shuffle coproduct $\Delta_{\shuffle}$ and where $ \triangleright $ is extended to $   \mathcal{P} \otimes S(\mathcal{P}) $ in the following way:
\begin{align*}
x \, \triangleright \, x_1\ldots x_k 
&=\sum_{i=1}^k x_1\ldots (x \, \triangleright \, x_i)\ldots x_k,
\end{align*}
with $ x_i \in  \mathcal{P} $. \label{star_0}
We now define the associative product $ \star $ as follows:
\begin{equs} \label{def_*}
w \star v = \sum_{(w)}\left(  (w^{(1)} \triangleright v) w^{(2)} \right) 
\end{equs}
Then, the associative product $\star$ on $S(\mathcal{P})$ is such that the Hopf algebra $(S(\mathcal{P}), \star, \Delta_{\shuffle})$ is isomorphic to the universal enveloping algebra $U(\mathcal{P})$ of the Lie algebra associated to $\mathcal{P}$, equipped with its standard Hopf-algebraic structure. Furthermore, the induced mapping from the category $\catname{PreLie}$ to the category $\catname{Hopf}$ of Hopf algebras is a functor. A morphism $\phi$ in $\catname{PreLie}$ is mapped to $S(\phi)$ where $S$ is the symmetric space functor.
\end{theorem}

Given a pre-Lie algebra $\mathcal{P}$ the Guin-Oudom procedure describes a way to impose a Hopf algebra structure on $S(\mathcal{P})$ by using the pre-Lie product to obtain an associative product on $S(\mathcal{P})$. This turns out to be isomorphic to $U(\mathcal{P})$. In fact, one obtains a functor from the category $\catname{PreLie}$ to the category $\catname{Hopf}$ of Hopf algebras, see the proof of Proposition 3.1 in \cite{Guin2}. Furthermore, Loday and Ronco, in \cite{LR} prove that this mapping is an equivalence of categories from $\catname{PreLie}$ to a certain category $\catname{CHA}$ of cofree-cocommutative right-sided combinatorial Hopf Algebras. For example, under this correspondence, the free pre-Lie algebra on one generator gives rise to the Grossman-Larson Hopf algebra. 

\begin{theorem}\label{Free pre-Lie -> Grossman - Larson} 
The grafting pre-Lie algebra with product $\curvearrowright$ and edge decorations from the set $E$ is the free pre-Lie algebra over $E$. Furthermore the product obtained by the Guin-Oudom construction above is the Grossman-Larson product and therefore $\CH_{\text{\tiny{GL}}} = (S(\CP_{E}^{V}), \star, \Delta_{\shuffle})$ is isomorphic to the universal enveloping algebra $U(\mathfrak{h})$, where $\mathfrak{h}$ is the Lie algebra induced by the grafting product $\curvearrowright$.
\end{theorem}

We also have analogous results for the deformed grafting product:

\begin{theorem}\label{Deformed Grafting0}
The Hopf algebra $\CH_{\text{\tiny{DGL}}}$ = $(S(\CP_E^V), \tilde{\star}, \Delta_{\shuffle})$, where $\tilde{\star}$ is obtained from $\widehat{\curvearrowright}$ via the Guin-Oudom construction given above, is isomorphic to the universal enveloping algebra $U(\mathfrak{g})$, where $\mathfrak{g}$ is the commutator Lie algebra induced by $\widehat{\curvearrowright}$.
\end{theorem}

We shall denote by $\CH_{\text{\tiny{DGL}}}$ the deformed Grossman-Larson Hopf algebra. This is due to the following theorem \cite[Theorem 3.4]{BM22}:

\begin{theorem}
The product $\tilde{\star}$ is dual to the deformed Butcher-Connes-Kreimer coproduct $\Delta_{\text{\tiny{DBCK}}}$.
\end{theorem}

 In \cite[Theorem 2.7]{BM22}, the authors prove that there exists an isomorphism $\Theta$ between the pre-Lie algebra $E_{\text{\tiny{GL}}}$ associated with the Grossman-Larson Hopf algebra and the pre-Lie algebra $E_{\text{\tiny{DGL}}}$ associated with the deformed Grossman-Larson Hopf algebra. One can describe recursively the isomorphism $\Theta$ by 
\begin{equs} \label{recursive_theta}
\begin{aligned}
\Theta\left( \mathcal{I}_a(X^k) \right) & = \mathcal{I}_a(X^k) \\
\Theta \left( \mathcal{I}_a(X^k \prod_{i=1}^n \mathcal{I}_{a_i}(\tau_i)) \right) 
&  = \prod_{i=1}^n \Theta \left(  \mathcal{I}_{a_i}(\tau_i)) \right) \, \widehat{\curvearrowright} \,  \mathcal{I}_a(X^k).
\end{aligned}
\end{equs}
Using the isomorphism $\Theta$, together with Theorem~\ref{Guin-Oudom}, we can now prove:

\begin{theorem} \label{theorem_functor}
The deformed Grossman-Larson Hopf algebra $\CH_{\text{\tiny{DGL}}}$ is isomorphic as a Hopf algebra to the original Grossman-Larson Hopf algebra $\CH_{\text{\tiny{GL}}}$.
\end{theorem}
\begin{proof}
Let $\Theta$ denote the isomorphism between the pre-Lie Algebras associated to $\CH_{\text{\tiny{GL}}}$ and $\CH_{\text{\tiny{DGL}}}$. We let $G: \catname{PreLie} \rightarrow \catname{Hopf}$ denote the Guin-Oudom functor. Then, $\Phi := G(\Theta)$ is an isomorphism between $\CH_{\text{\tiny{GL}}}$ and $\CH_{\text{\tiny{DGL}}}$ in the category of Hopf algebras.
\end{proof}

Then, our previously defined isomorphism $\Phi: \CH_{\text{\tiny{GL}}} \rightarrow \CH_{\text{\tiny{DGL}}}$ induces a linear isomorphism of the space of primitive elements for $\CH_{\text{\tiny{BCK}}}$, seen as a subspace of $\CH_{\text{\tiny{GL}}}$, onto its image. This allows us to prove the following theorem:

\begin{proposition} \label{prop_iso_tilde}
Let $\sigma \in \CH_{\text{\tiny{DGL}}}$. Then, there exists a subspace $\tilde{\CB} = \left\langle \sigma_{1}, \sigma_{2}, ... \right\rangle  \subseteq \CH_{\text{\tiny{DGL}}}$ such that 
\begin{equs}
\sigma = \sum_{R} \lambda_{R} \, \sigma_{r_{1}} \tilde{\star} ... \tilde{\star} \, \sigma_{r_{n}}
\end{equs}
with a unique such decomposition.
\end{proposition}

\begin{proof}
We pick the unique $\tau$ such that $\sigma = \Phi(\tau)$. We know that 
\begin{equs}
\tau = \sum_{R} \lambda_{R}\tau_{r_{1}} \star ... \star \tau_{r_{n}}
\end{equs}
 for some Butcher-Connes-Kreimer primitive elements $\tau_{r_{i}}$ belonging to $ B $. Hence, 
\begin{equs}
\sigma = \Phi(\tau) = \sum_{R} \lambda_{R}\Phi(\tau_{r_{1}}) \, \tilde{\star} ... \tilde{\star} \, \Phi(\tau_{r_{n}})
\end{equs}
We then pick $\sigma_{r_{i}} = \Phi(\tau_{r_{i}})$. This completes the proof by considering $ \tilde{\CB} =  \Phi(\CB)   $.
\end{proof}
From the previous proposition, given the basis $ B $, the basis one can use for $ \tilde{\star} $ is $ \Phi(B) = \left\lbrace \Phi(\tau_1), \Phi(\tau_2),...\right\rbrace  $. 
Then, one can exhibit an isomorphism $ \Psi_{\Phi} $ between the two spaces  $ S(\CP_{E}^{V}) $ and $ T(\Phi(\CB)) $ based on the basis of $ \Phi(\CB)$:
\begin{equs} \label{psi_phi}
\Psi_{\Phi} : \Phi(\tau_1) \ \tilde{\star} \ldots \tilde{\star} \ \Phi(\tau_r) \mapsto
\Phi(\tau_1) \otimes \ldots \otimes \Phi(\tau_r)
\end{equs}
where $ \Phi(\tau_1) \, \otimes \ldots \otimes \, \Phi(\tau_n) \in T(\Phi(\CB)) $. As a corollary, by composition of isomorphisms, we obtain:

\begin{theorem}\label{Main}
There exists a subspace $\CB = \langle\tau_1, \tau_2,...\rangle$ of $\CP_{E}^{V} $ such that $\CH_{\text{\tiny{DGL}}}$ is isomorphic as a Hopf algebra to the tensor Hopf algebra $(T(\CB), \otimes, \Delta_{\shuffle})$ endowed with the tensor product and the shuffle coproduct. 
\end{theorem}

\section{Extension of the Chapoton-Foissy isomorphism}

\label{section::4}

\ \ \ \ \ \ In this section, we prove our main result, Theorem~\ref{main_result_paper}, which asserts that the $\CH_{2}$ Hopf algebra, used in the context of regularity structures for recentering  the ensuing Taylor-type expansions around different points, is actually isomorphic to a simple quotient of the tensor Hopf algebra. This quotient comes from the Lie bracket between planted trees and extra elements $ X_i$ which are parts of $ \CH_2 $ but were absent in the previous section. The main difficulty is to check that the basis given by Theorem~\ref{Main} is stable under this quotient. This is proved in Proposition~\ref{stable_ideal} and relies on properties of the two derivations $\uparrow^i$ and $\mathcal{D}^i$. They commute with the isomorphism $\Psi_{\Phi}$, as shown in Proposition~\ref{conjecture_psi} and leave invariant the primitives of the Butcher-Connes-Kreimer Hopf algebra, see Corollary~\ref{primitive_BCK}. The formulation of the main result and its proof rely strongly on the post-Lie algebras introduced in \cite{BK} for describing $ \CH_2 $. Finally, we give a non-trivial extension of the Chapoton-Foissy isomorphism. This is a consequence of having better understood the two main algebraic components at play in the context of regularity structures currently used, which are the deformation and the post-Lie structure.

 We shall begin by introducing the concept of a post-Lie algebra, which generalizes that of a pre-Lie algebra. We also describe the recursive construction of an associative product on the universal envelope of a post-Lie algebra that directly generalizes the construction of Guin and Oudom. It was first introduced in \cite{ELM}.

\begin{definition} \label{definition_post_lie}
A post-Lie algebra is a Lie algebra $ (\mathfrak{g}, [.,.]) $ equipped with a bilinear product $ \triangleright $ satisfying the following identities:
\begin{equs} \label{ident1}
\begin{aligned}
x \triangleright [y,z] & = [x \triangleright y,z] + [y, x \triangleright z] \\
[x,y] \triangleright z & = a_{\triangleright}(x,y,z) - a_{\triangleright}(y,x,z)
\end{aligned}
\end{equs}
with $ x,y,x \in \mathfrak{g} $ and the commutator $ a_{\triangleright}(x,y,z) $ is given by:
\begin{equs}
a_{\triangleright}(x,y,z) = x \triangleright (  y  \triangleright z ) - (x \triangleright y) \triangleright z.
\end{equs}
\end{definition}
When $ (\mathfrak{g}, [.,.]) $ is the abelian Lie algebra, we obtain the notion of a pre-Lie algebra. One can define a new Lie bracket $[[ .,.]]$ given by:
\begin{equation} \label{bracket}
[[x,y]] = [x,y] + x \triangleright y - y \triangleright x.
\end{equation}
The post-Lie product $ \triangleright $  can be extended to a product on the universal enveloping algebra $ U(\mathfrak{g}) $ by first defining it on $ \mathfrak{g}  \otimes U(\mathfrak{g})$:
\begin{equs}
x \triangleright \one = 0, \quad x \triangleright y_1 ... y_n = \sum_{i=1}^n y_1 ... (x  \triangleright y_i) ... y_n.
\end{equs}
and then extending it to $ U(\mathfrak{g}) \otimes U(\mathfrak{g}) $ by defining:
\begin{equs}
\one  \triangleright A & = A, \quad
x A  \triangleright y = x \triangleright (A \triangleright y) -(x \triangleright A) \triangleright y, \\
A \triangleright B C & = \sum_{(A)} (A^{(1)} \triangleright B)(A^{(2)} \triangleright C).
\end{equs}
where $ A, B, C \in U(\mathfrak{g}) $ and $ x, y \in \mathfrak{g} $.  Here, $ (A) $ correspond to the deshuffle coproduct. One defines an associative product $*$ on $ U(\mathfrak{g})$, the universal enveloping algebra of $\mathfrak{g}$:
\begin{equs} \label{product_1}
A * B = \sum_{(A)} A^{(1)} (A^{(2)} \triangleright B).
\end{equs}
Then, a result that generalizes that of Guin and Oudom, allows us to exploit the underlying post-Lie structure on $\mathfrak{g}$ in order to gain additional insight on the structure of $U(\mathfrak{g})$. This is formalised in the following theorem:

\begin{theorem} \label{main_theorem_section_2}
The Hopf algebra $ (U(\mathfrak{g}),*,\Delta_{\shuffle}) $ is isomorphic to the enveloping
algebra $ U(\bar{\mathfrak{g}}) $ where $ \bar{\mathfrak{g}} $ is the Lie algebra equipped with the Lie bracket $ [[.,.]] $.
\end{theorem} 

This result has been used in \cite{BK}, in the context of regularity structures, in order to show that the $\star_{2}$ product, dual to the $\Delta_{2}$ coproduct appearing in \cite{BM22} and introduced in \cite{reg,BHZ}, comes directly from a post-Lie product by applying the above procedure. Below, we briefly recall this.    
We define the following spaces:
\begin{equs}
\mathcal{V} & = \Big \langle  \{ \CI_a(\tau), \, a \in \mathbb{N}^{d+1}, \, \tau \in \mathcal{T}^V_E \} \cup \{ X_{i} \}_{i = 0,..., d} \Big \rangle_{\mathbb{R}}, \\
\tilde{\mathcal{V}} & = \Big \langle  \CI_a(\tau), \, a \in \mathbb{N}^{d+1}, \, \tau \in \mathcal{T}^V_E \Big \rangle_{\mathbb{R}}.
\end{equs}
We denote by $\uparrow^{k}_{v}$ the operator acting on decorated trees by adding $k$ to the decoration of the node $v$.
We then define, for a tree $\tau \in \CT_{E}^{V}$ the operator $ \uparrow^{i} $ as follows:
\begin{equs}
\uparrow^{i} \tau  = \sum_{v \in N_{\tau}} \uparrow^{e_i}_v \tau.  
\end{equs}
This operator acts as a derivation on the multi-pre-Lie algebra of grafting products in the sense that:
\begin{equs} \label{derivation_node}
 \uparrow^{i} \left( \sigma \curvearrowright^a \tau  \right) =  (\uparrow^{i} \sigma) \curvearrowright^a  \tau +  \sigma \curvearrowright^a \, ( \uparrow^{i} \tau).
\end{equs}
 The derivation property \eqref{derivation_node} is not preserved under the deformation. One has the following identity similar to \cite[Proposition 4.4]{BK}.
\begin{equs} \label{non_commutation}
\uparrow^{i} \left( \sigma \widehat{\curvearrowright}^a \tau  \right) =   ( \uparrow^{i} \sigma  ) \, \widehat{\curvearrowright}^a \,   \tau  +   \sigma \, \widehat{\curvearrowright}^a \, ( \uparrow^{i} \tau)  - \sigma \, \widehat{\curvearrowright}^{a-e_i} \tau, 
\end{equs}
for all decorated trees $ \sigma, \tau $ and $ a \in \mathbb{N}^{d+1} $, $ i \in \lbrace 0,...,d \rbrace $. Looking at the above formula, one observes that the pair of operators $\tau \mapsto \sigma \widehat{\curvearrowright}^a \tau$ and $\tau \mapsto  \uparrow^{i} \tau$ does not satisfy the commutativity relation satisfied by the operators $\tau \mapsto \sigma \curvearrowright^{a} \tau$ and $\tau \mapsto  \uparrow^{i} \tau$. The non-commutative relation \ref{non_commutation} motivates the introduction of a Lie bracket together with a product that is a derivation for that bracket, that encode these relations in the form of a post-Lie algebra.
We begin by introducing a product $ \widehat{\triangleright} $ on $\mathcal{V}$:
\begin{equs} \label{prodcut_post_Lie}
X_i \, \widehat{\triangleright}  \,  \CI_{a}(\tau) & = \CI_{a}( \uparrow^i \tau), \quad \CI_{a}(\tau) \,  \widehat{\triangleright}  \, X_{i} = X_i \, \widehat{\triangleright}  \, X_{j} = 0
\\
\CI_{a}(\sigma) \, \widehat{\triangleright}  \, \CI_{b}(\tau) & = \CI_{a}(\sigma) \, \widehat{\curvearrowright} \, \CI_{b}(\sigma).
\end{equs}
In the sequel, we will use the notation $\uparrow^i \CI_{a}(  \tau) $ for $ \CI_{a}( \uparrow^i \tau) $.
We now proceed to define the appropriate Lie bracket, motivated by \ref{non_commutation}:

\begin{definition} \label{def_lie_trees}
We define the Lie bracket on $\mathcal{V}$ as $[x, y]_{0} = 0$ for $x, y \in \tilde{\mathcal{V}}$, $[x, y]_{0} = 0$ for $x, y \in \langle \ X_{i} \ \rangle_{\mathbb{R}}$ and as
\begin{equation}
[\CI_a(\tau),X_i]_0 =  \CI_{a-e_i}(\tau) 
\end{equation}
\end{definition}
With these definitions at hand, we have the following theorem (see \cite[Theorem 4.4]{BK}):
\begin{theorem}
The triple $( \mathcal{V}, [.,.]_{0}, \widehat{\triangleright})$ is a post-Lie algebra.
\end{theorem}
The bracket induced by the post-Lie algebra encodes all the (non-)commutativity relations between operators acting on decorated trees. However, most of these actually commute with one another, forming a pre-Lie algebra that lives inside the Lie algebra $\tilde{V}$. The extra post-Lie structure allows one to, roughly speaking, split the bracket into a commutative and non-commutative part. Hence the non-commutativity relations are actually encoded more succinctly by the Lie  bracket  $[.,.]_{0}$.

We denoted by $ U(\mathcal{V}_0) $ the enveloping algebra with the Lie bracket $ [.,.]_0 $ and  by $  U(\mathcal{V}) $ the enveloping algebra with the Lie bracket $ [[.,.]] $. We also set $ * $ to be the product obtained by the generalization of the Guin-Oudom procedure given in \eqref{product_1}. As a mere application of Theorem~\ref{main_theorem_section_2}, one gets

\begin{theorem} \label{main_result_trees}
The Hopf algebra $U(\mathcal{V})$ is isomorphic to the Hopf algebra $(U(\mathcal{V}_0), *, \Delta_{\shuffle})$.
\end{theorem}

Then, the main result of \cite{BK} is

\begin{theorem}
The Hopf algebra $(U(\mathcal{V}_0), *, \Delta_{\shuffle})$ is isomorphic to the Hopf algebra $\CH_{2} = (\CT_{E}^{V}, \star_{2}, \Delta_{\shuffle})$ as presented in \cite{BM22}.
\end{theorem}

\begin{remark}
An explicit formula for the $\star_{2}$ product for $\sigma = X^{k} \prod_{i \in I} \CI_{a_{i}}(\sigma_{i})$ and $\tau \in \CT_{E}^{V}$ is given by
\begin{equs}
\CI_{b}(\sigma \star_{2} \tau) := \tilde{\uparrow}_{N_{\tau}}^{k} \( \prod_{i \in I} \CI_{a_{i}}(\sigma_{i}) \, \widehat{\curvearrowright} \, \CI_{b}(\tau) \),
\quad
\tilde{\uparrow}_{N_{\tau}}^{k} = \sum_{k = \sum_{v \in N_{\tau}}k_{v} } \uparrow_{v}^{k_{v}}
\end{equs}
\end{remark}

We shall now decompose trees of the form $ X^k \prod_i \mathcal{I}_{a_i}(\tau_i) $ with decoration $ k $ at the root. So far, we have been successful in doing this for trees with no root decoration. For these terms, we will need to utilize the underlying post-Lie structure and the fact that $X_{i}$ does not commute with any term of the form $ \mathcal{I}_{a}(\tau) $. Instead one has:
\begin{equs}
X_i \star_2 \mathcal{I}_{a}(\tau) -  \mathcal{I}_{a}(\tau) \star_2 X_i = \uparrow_{e_i} \mathcal{I}_{a}(\tau)  - \mathcal{I}_{a - e_i}(\tau)
\end{equs}
where $ \star_2 $ is the product constructed from the post-Lie product. The restriction of this product on the space spanned by planted trees coincides with $ \tilde{\star} $. What we obtain will then be an isomorphism with a space of words quotiented by the following relation:
\begin{equs}
 X_i \otimes \mathcal{I}_{a}(\tau) -  \mathcal{I}_{a}(\tau) \otimes X_i = \uparrow_{e_i} \mathcal{I}_{a}(\tau)  - \mathcal{I}_{a - e_i}(\tau)
\end{equs}
where now the trees with a single node are treated as letters. Let us explain how this works for a decorated tree of the form $ X^k \prod_{i=1}^n \CI_{a_i}(\tau_i) $ when one wants to decompose the following terms:
 \begin{equs}
 X^k \prod_{i=1}^n \CI_{a_i}(\tau_i) \star_2 \tau.
 \end{equs} 
We begin by making the following remarks that shall prove useful in what follows:
\begin{itemize}
\item By choosing a different ordering in the Poincare-Birkhoff-Witt theorem, we clearly see that the set of elements of the form
\begin{equs}
\prod_{i=1}^n \CI_{a_i}(\tau_i)X^{k}
\end{equs}
where $\textbf{F} = \CI_{a_{1}}(\tau_{1}) \cdot \cdot \cdot \CI_{a_{n}}(\tau_n) $ ranges over all forests of planted trees and $\textbf{m} \in \mathbb{N}^{d+1}$ is a basis for $U(\mathcal{V}_0)$.

\vspace{10pt}

\item The operator $\tau \mapsto 
\prod_{i=1}^n \CI_{a_i}(\tau_i)X^{k} \star_{2} \tau$ 
is equal to the operator $\tau \mapsto 
\prod_{i=1}^n \CI_{a_i}(\tau_i) \star_{2} X^{k} \star_{2} \tau$
\end{itemize}
We introduce a second derivation $ \mathcal{D}^{i} $ defined by
\begin{equs}
\mathcal{D}^{i} \tau = \sum_{e \in E_\tau}
\mathcal{D}^{i}_e \tau
\end{equs}
where $ \mathcal{D}^{i}_e $ adds $ -e_i$ to the decoration of the edge $ e $ if possible. Otherwise, it is equal to zero.

\begin{proposition} \label{conjecture_psi}
For every $ \tau \in \mathcal{P}^V_E $, one has:
\begin{equs}
\uparrow^{i} \Phi(\tau) = \Phi( \uparrow^{i} \tau ) - \Phi( \mathcal{D}^{i} \tau)
\end{equs}
\end{proposition}
\begin{proof}
We first consider a decorated tree $ \tau $ of the form 
\begin{equs}
\tau = \tau_{1} \curvearrowright^{b} \mathcal{I}_a(\tau_{2}) 
\end{equs}
then one has
\begin{equs}
\uparrow^{i} \Phi(\tau) & = 
\uparrow^{i} \left( \Phi(\tau_1) \, \widehat{\curvearrowright}^{b} \, \Phi(\mathcal{I}_a(\tau_{2}))  \right)
\\ & = ( \uparrow^{i} \Phi(\tau_1) ) \, \widehat{\curvearrowright}^{b} \, \Phi( \mathcal{I}_a( \tau_{2})) +  \Phi(\tau_1) \, \widehat{\curvearrowright}^{b} \, (\uparrow^{i}  \Phi(\mathcal{I}_a( \tau_{2}))) \\ & -  \Phi(\tau_1)  \, \widehat{\curvearrowright}^{b-e_i} \, \Phi(\mathcal{I}_a( \tau_{2})) 
\end{equs}
where we have used \eqref{non_commutation}. Then, one can apply an induction hypothesis on $ \mathcal{I}_b(\tau_1) $ and $ \mathcal{I}_a(\tau_2)  $ and one gets
\begin{equs}
 \uparrow^{i} \Phi(\mathcal{I}_b(\tau_1))   & = \Phi(\mathcal{I}_b( \uparrow^{i} \tau_1 )) - \Phi( \mathcal{I}_b(\mathcal{D}^{i} \tau_1)) \\ 
 \uparrow^{i} \Phi(\mathcal{I}_a(\tau_2))   & = \Phi(\mathcal{I}_a( \uparrow^{i} \tau_2 )) - \Phi( \mathcal{I}_a(\mathcal{D}^{i} \tau_2))
\end{equs}
Then, one observes that
\begin{equs}
\Phi( \mathcal{I}_a(\mathcal{D}^{i} \tau)) &  = \Phi(\tau_1  \curvearrowright^{b-e_i} \mathcal{I}_a(\tau_{2})) + \Phi(  \mathcal{D}^i \tau_1  \curvearrowright^{b} \mathcal{I}_a(\tau_{2})) + \Phi(  \tau_1 \curvearrowright^{b} \mathcal{I}_a(\mathcal{D}^i \tau_{2}))
\\  \Phi(\mathcal{I}_a( \uparrow^{i} \tau )) & = \Phi( \uparrow^{i} \tau_1  \curvearrowright^{b} \mathcal{I}_a(\tau_{2})) + \Phi(  \tau_1  \curvearrowright^{b} \mathcal{I}_a( \uparrow^{i} \tau_{2}))
\end{equs}
We conclude by using again the morphism property of $ \Phi $ that gives us for example:
\begin{equs}
\Phi( \uparrow^{i} \tau_1  \curvearrowright^{b} \mathcal{I}_a(\tau_{2})) = \Phi( \uparrow^{i} \CI_b(\tau_1)) \,  \widehat{\curvearrowright} \, \Phi(\mathcal{I}_a(\tau_{2}))
\end{equs}
 and the fact that $ \tau $ is generated by the family $ (\widehat{\curvearrowright}^{b})_{b} $.
\end{proof}

\begin{proposition} \label{commutattion_derivations}
One has the following commutation identities:
\begin{equs}
\Delta_{\text{\tiny{BCK}}}\uparrow^i & =  \left( \uparrow^i  \otimes \, \one  \right) \Delta_{\text{\tiny{BCK}}} + \left( \one \, \otimes \uparrow^i    \right) \Delta_{\text{\tiny{BCK}}} 
\\ \Delta_{\text{\tiny{BCK}}}\mathcal{D}^i & =  \left( \mathcal{D}^i  \otimes \one  \right) \Delta_{\text{\tiny{BCK}}} + \left( \one \otimes \mathcal{D}^i    \right) \Delta_{\text{\tiny{BCK}}}  
\end{equs}
with the convention that $ \mathcal{D}^i \one = \uparrow^i \one = 0 $. 
\end{proposition}
\begin{proof}
This is just a consequence of the fact that $ \uparrow^i $ and $ \mathcal{D}^i $ are derivations for $ \curvearrowright^a $ and therefore for the Grossman-Larson product $ \star $. By going to the dual, one gets the desired identities.
\end{proof}

\begin{corollary} \label{primitive_BCK}
The set of primitives elements for $\CH_{\text{\tiny{BCK}}}$ is stable under the action of the derivations $\uparrow^{e_{i}}$ as well as the derivations $\mathcal{D^i}$.
\end{corollary}
\begin{proof}
Let $ \tau $ a primitive elements, one has
\begin{equs}
\Delta_{\text{\tiny{BCK}}} \uparrow^i \tau =  \left( \uparrow^i  \otimes \, \one  \right) \Delta_{\text{\tiny{BCK}}} \tau + \left( \one \, \otimes \uparrow^i    \right) \Delta_{\text{\tiny{BCK}}} \tau
\end{equs}
where we have used Proposition~\ref{commutattion_derivations} . Then, from the primitiveness of $ \tau $
\begin{equs}
\Delta_{\text{\tiny{BCK}}}  \tau 
= \tau \otimes \one + \one \otimes \tau
\end{equs}
which allows us to get using the fact that $ \uparrow^i \one = 0  $:
\begin{equs}
\Delta_{\text{\tiny{BCK}}} \uparrow^i \tau = \uparrow^i \tau  \otimes \one +
\one \, \otimes \uparrow^i \tau
\end{equs}
The proof works as the same for $ \mathcal{D}^i $.
\end{proof}


\begin{proposition} \label{stable_ideal}
If $\sigma = \Psi(\tau)$ for some primitive element $\tau$ with respect to the $\Delta_{\text{\tiny{BCK}}}$ coproduct, then $\uparrow^{e_i}\sigma$  and $\mathcal{D}^i \sigma$ are also in the image of $Prim(\CH_{\text{\tiny{BCK}}})$.
\end{proposition}
\begin{proof}
This is a  consequence of Proposition~\ref{conjecture_psi} and Corollary~\ref{primitive_BCK}.
\end{proof}
We can now state and prove our main result:
 \begin{theorem} \label{main_result_paper} We equip $\mathcal{T}^V_E$ with two products: $ \tilde{\star} $ is the product dual to the deformed Butcher-Connes-Kreimer coproduct and $ \star_2 $ is the product of $\CH_{2}$.
 We let $W$ be the linear space of the words from the alphabet $ A $ whose letters are the $X_i$ and $ \Phi(\CI_{a}(\tau))$ where $ \CI_a(\tau) $ is a primitive element for $ \Delta_{\text{\tiny{BCK}}} $ and belongs to $ B $ given in \eqref{basis_B}.
 We define $ \tilde{W} $ as the quotient of $W$ by the Hopf ideal $ \CJ $ generated by the elements
\begin{equs}
  \{ X_i \otimes \Phi( \mathcal{I}_{a}(\tau) ) - \Phi( \mathcal{I}_{a}(\tau) ) \otimes X_i - \uparrow^{i} \Phi( \mathcal{I}_{a}(\tau))  - \Phi(\mathcal{I}_{a - e_i}(\tau)) \}
\end{equs}
where $ \mathcal{I}_{a}(\tau)  \in B $.
 Then, there exists a Hopf algebra isomorphism $ \Psi_{\Phi} $ (extension of $\Psi_{\Phi}$ defined in \eqref{psi_phi}) between $ \mathcal{T}_E^V $ equipped with $ \star_2 $ and the deshuffle coproduct and $ \tilde{W} $  equipped with the concatenation coproduct and the deshuffle coproduct.
 The map $ \Psi_{\Phi} $ is given by
\begin{equs}
 \Psi_{\Phi} : \prod_{i=1}^n \CI_{a_i}(\tau_i) X^k \mapsto \Psi_{\Phi}(\prod_{i=1}^n \CI_{a_i}(\tau_i)) \otimes_{j=0}^{d} \otimes_{i=1}^{k_j} X_j.
\end{equs}
\end{theorem} 
 
\begin{proof}
We first apply the isomorphism described in Proposition~\ref{prop_iso_tilde}  on $\sigma = \prod_{i=1}^n \CI_{a_i}(\tau_i)$ by writing
\begin{equs}
\prod_{i=1}^n \CI_{a_i}(\tau_i) = \sum_{R} \lambda_{R} \, \sigma_{r_{1}} \star_{2} ... \star_{2} \sigma_{r_{n}} 
\end{equs}
with $ \sigma_{r_i} \in A $. Here, we have used the fact that $ \star_2$  and $\tilde{\star}$ coincide on planted decorated trees. We then map $  \prod_{i=1}^n \CI_{a_i}(\tau_i) X^k $ as follows:
\begin{equs}
\prod_{i=1}^{n}  \CI_{a_{i}}(\tau_i) X^k \mapsto \sum_{R} \lambda_{R} \sigma_{r_{1}} \otimes ... \otimes \sigma_{r_{n}} \otimes X_{0}^{\otimes k_{0}} \otimes ... \otimes X_{d}^{\otimes k_{d}}
\end{equs}
By virtue of Proposition~\ref{stable_ideal}, this clearly gives an isomorphism onto the Hopf algebra $\tilde{W}$. Indeed, given a letter $ \Phi(\mathcal{I}_a(\tau)) $, one has that $ \uparrow^i \Phi(\mathcal{I}_a(\tau)) $ and $ \Phi(\mathcal{I}_{a - e_i}(\tau)) $ are linear combination of letters of $ W $.
\end{proof}

 \section{Applications in regularity structures}
 
 \label{section::5}
 
 In this section we restrict ourselves to the setting that is more specific to the theory of regularity structures, specifically the structures first appearing in the works \cite{reg,BHZ}. For an introduction to the theory see \cite{FrizHai,EMS,BaiHos}. This involves considering a Hopf subalgebra of the $\CH_{2}$ Hopf algebra, that consists of trees with branches of positive degree. We shall use the theorem proved in the previous section to embed this into the tensor Hopf algebra. This allows for an encoding of the iterated integrals appearing when solving the equations, in the form of words. We begin by defining the space:
\begin{equs}
T_{+} := \{ X^{k}\prod_{i=1}^{n} \CI_{a_{i}}(\tau_{i}) \ | \ \alpha( \CI_{a_{i}}(\tau_{i})) > 0, \, \tau \in T^V_E \}
\end{equs}
We also define $\CT_{+}$ to be the linear span of $T_{+}$. Here, $ \alpha $ is a degree map computing a number associated to a decorated tree. This corresponds of some kind of regularity of the stochastic integral associated to the decorated tree. It takes into account the decoration on the edges that could both encode distributional noises or convolution with kernel that provides a smoothing effect via Schauder estimates. We refrain to give a precise definition that could be found in many works \cite{reg,BHZ}.

For each subcritical singular SPDE one constructs a Hopf subalgebra $\CT^+_{R}$ of $\CT_{+}$ by attaching a generating rule $R$ to the nonlinearity $F$ of the equation. The rule induces a recursive procedure that generates the entire Hopf subalgebra $\CT_{R}^+$.  This procedure may be thought of as formal Picard iteration. The resulting Hopf subalgebra is then used to describe the regularity structure for the given equation. In the next theorem, we denote by $ \cdot $ the product on $ \CT_+ $. 

\begin{theorem}
\label{T_+ isomorphism}
The Hopf algebra  $(\CT_{+}, \star_{2}, \Delta_{\shuffle})$, which is the graded dual of $(\CT_{+}, \cdot, \Delta_{2})$, is isomorphic to a Hopf subalgebra of $T(A)/\CJ$.

\end{theorem}
 
\begin{proof}

By Theorem~\ref{main_result_paper}, we have an isomorphism $\Phi: \CH_{2} \rightarrow T(A)/\CJ$. By simply restricting $\Phi$ to $\CT_{+}$ we obtain a Hopf algebra ismorphism of $\CT_{+}$ onto its image.
\end{proof}
Let explain how this algebraic result allows to interpret regularity structures as some kind of geometric rough paths.
 Solutions $ u $ of local subcritical singular stochastic partial equations (SPDEs) are locally described by
\begin{equs}
u(y) -  u(x) = \sum_{\tau \in \mathcal{T}_R} u_{\tau}(x)(\Pi_x \tau)(y), \quad (\Pi_x \tau)(y) \lesssim |y-x|^{\alpha(\tau)}
\end{equs}
where $ x,y \in \mathbb{R}^{d+1} $, $ \mathcal{T}_R $ are the decorated planted trees generated by the rule $R$, $ (\Pi_x \tau)(y) $ are stochastic iterated integrals recentered around the point $ x $ such that one has a behaviour close to $ x $ according to the degree of the given decorated tree. The $ u_{\tau}(x) $ are some kind of derivatives.    Then, the theory of regularity structures provides a reexpansion map $ \Gamma_{xy} $ that allows us to move the recentering:
\begin{equs}
\Pi_y = \Pi_x \Gamma_{xy}.
\end{equs}
The collection of these two maps $ (\Pi_x,\Gamma_{xy}) $ is what is referred to as a model \cite[Def. 3.1]{reg}. One important algebraic construction is to represent $ \Gamma_{xy} $ via a character $ \gamma_{xy} : \CT_+  \rightarrow \mathbb{R}$  multiplicative for the tree product. This description is given via a co-action $ \Delta : \CT \rightarrow \CT \otimes \CT_+ $
\begin{equs} \label{chen_RS}
\Gamma_{xy}  = \left( \id \otimes \gamma_{xy} \right) \Delta, \quad | \gamma_{xy}(\tau) | \lesssim | y-x |^{\alpha(\tau)}.
\end{equs}
The character $ \gamma_{xy} $ can be viewed as an extension of  branched rough paths to the multidimentional case as $ x, y \in \mathbb{R}^{d+1} $. Moreover, it satisfies some Chen's relation:
\begin{equs}
\gamma_{xy} \star_2 \gamma_{yz} = \gamma_{xz}
\end{equs}
 We denote the space of such maps by $ \bf{TM}^{\alpha} $ called $  \alpha$-Tree-indexed Models. Maps $ \gamma_{xy} $ defined as character on $ \Psi(\CT_+) $ are $ \alpha$-Geometric Models denoted by $ \bf{GM}^{\alpha} $. They satisfy the following properties:
\begin{equs} \label{def_GM}
\gamma_{xy} \otimes \gamma_{yz} = \gamma_{xz}, 
\quad | \gamma_{xy}(\Psi(\tau)) | \lesssim | y-x |^{\alpha(\tau)}.
 \end{equs} 
  We could have used the terminology of anisotropic rough paths but the characters are defined on a quotient of a tensor Hopf algebra and not the tensor Hopf algebra itself.  One can rephrase our main algebraic result as:
\begin{theorem} \label{chiso_bis} Let $ X \in \bf{TM}^{\alpha} $, then $ \hat{X} := \Psi(X) \in \bf{GM}^{\alpha}$.
\end{theorem}
\begin{proof}
The analytical bounds are easily satisfied by realising that:
\begin{equs}
\langle  \Psi(X)_{xy}, \Psi(\tau) \rangle
= \langle  X_{xy}, \tau \rangle.
\end{equs}
The algebraic identities are  such as Chen's relation are preserved by the map $ \Psi $.
\end{proof}

\begin{remark} As in \cite{Br20}, one can investigate the action of the renormalisation on this construction by looking at maps $ M $ that are morphisms for the product $ \star_2 $ which are BPHZ renormalisation maps. One of the main issue is that $ \CT_+ $ may not be stable under $ M $ due to the constraint imposed on the degree being positive. Extended decorations on trees have been introduced in \cite{BHZ} in order to guarantee that $ M $ is degree preserving. This property implies that $ \CT_+ $ is invaraint under $ M $. One can easily check that $ M $ commutes with $ \Phi $ and then it is possible to find a map $ \tilde{M} $ defined on $ T(A)/\CJ $ such that it will commute with $ \Psi $:
\begin{equs}
\tilde{M} \Psi = \Psi M.
\end{equs}
This will be an equivalent of Theorem 4.7 in \cite{Br20}.
\end{remark}


\begin{thebibliography}{Cha10}
\expandafter\ifx\csname url\endcsname\relax
  \def\url#1{\texttt{#1}}\fi
\expandafter\ifx\csname urlprefix\endcsname\relax\def\urlprefix{URL }\fi
\expandafter\ifx\csname href\endcsname\relax
  \def\href#1#2{#2}\fi
\expandafter\ifx\csname burlalt\endcsname\relax
  \def\burlalt#1#2{\href{#2}{\texttt{#1}}}\fi


 
 \bibitem{BB21b}
I.~Bailleul, Y.~{Bruned}.
\newblock { \em  Parametrization of renormalized models for singular stochastic PDEs}. 
 \burlalt{arXiv:2106.08932}{https://arxiv.org/abs/2106.08932}. 
 

 



\bibitem{BC}
{\rm H.~Boedihardjo, I.~Chevyrev}.
\newblock \emph{An isomorphism between branched and geometric rough paths.}
\newblock Ann. Inst. H. Poincaré Probab. Statist. \textbf{55}, no.~2,
  (2019), 1131--1148.
\newblock
  \burlalt{doi:10.1214/18-AIHP912}{https://dx.doi.org/10.1214/18-AIHP912}.
  
\bibitem{BCCH}
 { \rm Y. Bruned, A. Chandra, I. Chevyrev,
  M. Hairer}.
\newblock {\em Renormalising SPDEs in regularity structures}.
\newblock J. Eur. Math. Soc. (JEMS), \textbf{23}, no.~3, (2021), 869--947.
\newblock
  \burlalt{doi:10.4171/JEMS/1025}{http://dx.doi.org/10.4171/JEMS/1025}.  
  
  \bibitem{Br173}
{\rm Y.~Bruned, C.~Curry,  K.~Ebrahimi-Fard}.
\newblock \emph{Quasi-shuffle algebras and renormalisation of rough differential
  equations}.
\newblock B. Lond. Math. Soc. \textbf{52}, no.~1, (2020), 43--63.
  \burlalt{doi:10.1112/blms.12305}{https://dx.doi.org/10.1112/blms.12305}.

\bibitem{BCF}
{\rm Y.~Bruned, I.~Chevyrev, P.~K. Friz}.
\newblock \emph{Examples of renormalized sdes}.
\newblock Stochastic Partial Differential Equations and Related
  Fields,  303--317. Springer, 2018.
\newblock
  \burlalt{doi:10.1007/978-3-319-74929-7_19}{https://dx.doi.org/10.1007/978-3-319-74929-7_19}.

\bibitem{BCFP}
{\rm Y.~{Bruned}, I.~{Chevyrev}, P.~K. {Friz}, R.~{Preiss}}.
\newblock \emph{A rough path perspective on renormalization}.
\newblock J. Funct. Anal. \textbf{277}, no.~11, (2019), 108283.
\newblock
  \burlalt{doi:10.1016/j.jfa.2019.108283}{https://dx.doi.org/10.1016/j.jfa.2019.108283}.
  
\bibitem{BaiHos}
I.~{Bailleul}, M.~{Hoshino}.
\newblock {\em A tourist's guide to regularity structures.}
\newblock \burlalt{arXiv:2006.03524}{https://arxiv.org/abs/2006.03524}.
  

\bibitem{BHZ}
{\rm Y. Bruned, M. Hairer, L. Zambotti}.
\newblock {\em Algebraic renormalisation of regularity structures.}
\newblock Invent. Math. \textbf{215}, no.~3, (2019), 1039--1156.
\newblock
  \burlalt{doi:10.1007/s00222-018-0841-x}{https://dx.doi.org/10.1007/s00222-018-0841-x}.
  
  \bibitem{EMS}
  {\rm Y. Bruned, M. Hairer, L. Zambotti}.
\newblock {\em Renormalisation of Stochastic Partial Differential Equations.}
\newblock EMS Newsletter \textbf{115}, no.~3, (2020), 7--11.
\newblock
  \burlalt{doi: 10.4171/NEWS/115/3}{http://dx.doi.org/10.4171/NEWS/115/3}.
 
 \bibitem{BK1}
{\rm Y.~Bruned, F.~Katsetsiadis}. \newblock {\em Ramification of Volterra-type Rough Paths}. Electron. J. Probab. \textbf{28}, (2023), 1--25.
 \burlalt{doi:10.1214/22-EJP890}{https://www.doi.org/10.1214/22-EJP890}. 

 
 \bibitem{BK}
{\rm Y.~Bruned, F.~Katsetsiadis}. \newblock {\em Post-Lie algebras in Regularity Structures}. 
 Forum of
Mathematics, Sigma, 11, e98, (2023), 1–20. 
\newblock \burlalt{doi:10.1017/fms.2023.93}{https://www.doi.org/10.1017/fms.2023.93}. 

 
 \bibitem{BM22}
Y.~Bruned, D.~Manchon.
\newblock {\em Algebraic deformation for (S)PDEs}. J. Math. Soc. Japan. {\bf{75}}, no.~2, (2023), 485-526.
\newblock 
  \burlalt{doi:10.2969/jmsj/88028802}{http://dx.doi.org/10.2969/jmsj/88028802}.  
  
 


\bibitem{BP}
N.~N. Bogoliubow, O.~S. Parasiuk.
\newblock { \em \"{U}ber die {M}ultiplikation der {K}ausalfunktionen in der
  {Q}uantentheorie der {F}elder.}
\newblock Acta Math. \textbf{97}, (1957), 227--266.
\newblock
  \burlalt{doi:10.1007/BF02392399}{http://dx.doi.org/10.1007/BF02392399}.

  
   \bibitem{Br20}
 Y.~{Bruned}.
 \emph{Renormalisation from non-geometric to geometric rough paths}. 
\newblock \\ Ann. Inst. H. Poincaré Probab. Statist., \textbf{58}, no.~2, (2022), 1078-1090. 
 \burlalt{doi:10.1214/21-AIHP1178}{http://dx.doi.org/10.1214/21-AIHP1178}.  


  
  \bibitem{BS}
Y.~{Bruned}, K.~{Schratz}. \newblock { \em Resonance based schemes for dispersive equations via decorated 
        trees}. Forum of Mathematics, Pi, 10, E2. 
\newblock \burlalt{doi:10.1017/fmp.2021.13}{https://doi.org/10.1017/fmp.2021.13}.
  
  \bibitem{Butcher}
{ \rm J. C. Butcher}.
\newblock {\em An algebraic theory of integration methods.}
\newblock Math. Comp. \textbf{26}, (1972), 79--106.
\newblock \burlalt{doi:10.2307/2004720}{http://dx.doi.org/10.2307/2004720}.
  

\bibitem{cha10}
{\rm F.~Chapoton}.
\newblock \emph{Free pre-{Lie} algebras are free as {Lie} algebras}.
\newblock Canadian Mathematical Bulletin \textbf{53}, no.~3, (2010),
  425–437.
\newblock
  \burlalt{doi:10.4153/CMB-2010-063-2}{https://dx.doi.org/10.4153/CMB-2010-063-2}.
  
  \bibitem{CH16}
A.~Chandra, M.~Hairer.
\newblock {\textsl{An analytic {BPHZ} theorem for regularity structures.}}
\newblock \burlalt{arXiv:1612.08138}{http://arxiv.org/abs/1612.08138}.

\bibitem{CK1}
A.~Connes, D.~Kreimer.
\newblock { \em Hopf algebras, renormalization and noncommutative geometry.}
\newblock Comm. Math. Phys. \textbf{199}, no.~1, (1998), 203--242.
\newblock
  \burlalt{doi:10.1007/s002200050499}{http://dx.doi.org/10.1007/s002200050499}.

\bibitem{CK2}
A.~Connes, D.~Kreimer.
\newblock { \em Renormalization in quantum field theory and the {R}iemann-{H}ilbert
  problem {I}: the {H}opf algebra structure of graphs and the main theorem.}
\newblock  Comm. Math. Phys. \textbf{210}, (2000), 249--73.
\newblock
  \burlalt{doi:10.1007/s002200050779}{http://dx.doi.org/10.1007/s002200050779}.

\bibitem{ELM}
  {\rm K.~Ebrahimi-Fard, A.~Lundervold, H.~Munthe-Kaas}.
\newblock {\em On the {L}ie enveloping algebra of a post-{L}ie algebra.}
\newblock J. Lie Theory \textbf{25}, no.~4, (2015), 1139--1165.

\bibitem{FrizHai}
P.~K. {Friz}, M.~{Hairer}.
\newblock \emph{{A Course on Rough Paths}}.
\newblock Springer International Publishing, 2020.
\newblock
  \burlalt{doi:10.1007/978-3-030-41556-3}{https://dx.doi.org/10.1007/978-3-030-41556-3}.


\bibitem{Foi02}
{ \rm L.~Foissy}.
\newblock \emph{Finite dimensional comodules over the hopf algebra of rooted trees}.
\newblock J. Algebra \textbf{255}, no.~1, (2002), 89 -- 120.
\newblock
  \burlalt{doi:10.1016/S0021-8693(02)00110-2}{https://dx.doi.org/10.1016/S0021-8693(02)00110-2}.
  
  \bibitem{Foi21}
{ \rm L.~Foissy}.
\newblock \emph{Algebraic structures on typed decorated rooted trees}.
\newblock SIGMA \textbf{17}, no.~86, (2021), 1-- 28.
\newblock
  \burlalt{doi:10.3842/SIGMA.2021.086}{https://dx.doi.org/10.3842/SIGMA.2021.086}.
  
  \bibitem{GL}
 R. Grossman, R. G. Larson,
\newblock {\em Hopf algebraic structure of families of trees},
\newblock J. Algebra \textbf{126}, no.~1 (1989), 184--210. \burlalt{doi:10.1016/0021-8693(89)90328-1}{https://doi.org/10.1016/0021-8693(89)90328-1}.
  
\bibitem{Guin1}
D. Guin,  J. M. Oudom, \emph{Sur l'alg\`ebre enveloppante d'une
  alg\`ebre pr\'{e}-{L}ie}, C. R. Math. Acad. Sci. Paris \textbf{340} (2005),
  no.~5, 331--336.
  \burlalt{doi:10.1016/j.crma.2005.01.010}{https://doi.org/10.1016/j.crma.2005.01.010}. 
  

\bibitem{Guin2}
D. Guin,  J. M. Oudom, \emph{On the {L}ie enveloping algebra of a pre-{L}ie algebra}, J.
  K-Theory \textbf{2} (2008), no.~1, 147--167.
  \burlalt{doi:10.1017/is008001011jkt037}{https://doi.org/10.1017/is008001011jkt037}. 
  
  
  \bibitem{Gubinelli2004}
{\rm M.~Gubinelli}.
\newblock \emph{Controlling rough paths}.
\newblock J. Funct. Anal. \textbf{216}, no.~1, (2004), 86
  -- 140.
\newblock
  \burlalt{doi:10.1016/j.jfa.2004.01.002}{https://dx.doi.org/10.1016/j.jfa.2004.01.002}.

\bibitem{Gub06}
{\rm M.~Gubinelli}.
\newblock \emph{Ramification of rough paths}.
\newblock J. Differ. Equ. \textbf{248}, no.~4, (2010), 693 -- 721.
  \burlalt{doi:10.1016/j.jde.2009.11.015}{https://dx.doi.org/10.1016/j.jde.2009.11.015}.
 
 \bibitem{reg}
{\rm M. Hairer}.
\newblock {\em A theory of regularity structures.}
\newblock Invent. Math. \textbf{198}, no.~2, (2014), 269--504.
\newblock
  \burlalt{doi:10.1007/s00222-014-0505-4}{https://dx.doi.org/10.1007/s00222-014-0505-4}. 
   
     
 \bibitem{Hepp}
K. ~Hepp.
\newblock {\em On the equivalence of additive and analytic renormalization}.
\newblock Comm. Math. Phys. \textbf{14}, (1969), 67--69.
\newblock \burlalt{doi:10.1007/
BF01645456}{http://dx.doi.org/10.1007/
BF01645456}.

\bibitem{HK15}
{\rm M.~Hairer, D.~Kelly}.
\newblock \emph{Geometric versus non-geometric rough paths}.
\newblock Ann. Inst. H. Poincaré Probab. Statist. \textbf{51}, no.~1,
  (2015), 207--251.
\newblock
  \burlalt{doi:10.1214/13-AIHP564}{https://dx.doi.org/10.1214/13-AIHP564}.
  
\bibitem{LOT}
P.~Linares, F.~Otto, M.~Tempelmayr.
\newblock {\em The structure group for quasi-linear equations via universal enveloping algebras}. 
\newblock \burlalt{arXiv:2103.04187
}{https://arxiv.org/abs/2103.04187}. 
 


\bibitem{LR}
J.-L.~Loday, M. Ronco.
{\em Combinatorial Hopf algebras.}  Quanta of maths, 347–383,
Clay Math. Proc., 11, Amer. Math. Soc., Providence, RI, 2010. 
  
  \bibitem{LV07}
{ \rm T.~Lyons, N.~Victoir}.
\newblock \emph{An extension theorem to rough paths}.
\newblock Ann. Inst. Henri Poincare (C) Anal. Non Lineaire
  \textbf{24}, no.~5, (2007), 835 -- 847.
\newblock
  \burlalt{doi:10.1016/j.anihpc.2006.07.004}{https://dx.doi.org/10.1016/j.anihpc.2006.07.004}.

\bibitem{Lyons98}
{\rm T.~J. Lyons}.
\newblock \emph{Differential equations driven by rough signals}.
\newblock Rev. Mat. Iberoamericana \textbf{14}, no.~2, (1998), 215--310.
\newblock \burlalt{doi:10.4171/RMI/240}{https://dx.doi.org/10.4171/RMI/240}.

\bibitem{Murua2017}
{\rm A.~Murua, J.~M. Sanz-Serna}.
\newblock \emph{Word series for dynamical systems and their numerical integrators}.
\newblock Found. Comput. Math. \textbf{17}, no.~3,
  (2017), 675--712.
\newblock
  \burlalt{doi:10.1007/s10208-015-9295-3}{https://dx.doi.org/10.1007/s10208-015-9295-3}.
  
  \bibitem{Murua2006}
{\rm A.~Murua}.
\newblock \emph{The Hopf Algebra of Rooted Trees, Free Lie Algebras, and Lie Series}.
\newblock Found. Comput. Math. \textbf{17}, no.~6,
  (2006), 387–426.
\newblock
  \burlalt{doi:10.1007/s10208-003-0111-0}{https://dx.doi.org/10.1007/s10208-003-0111-0}.
  
  

\bibitem{OSSW}
 F.~Otto, J.~Sauer, S.~Smith, H.~Weber.
\newblock {\em A priori bounds for quasi-linear SPDEs in the full sub-critical regime}. 
\newblock \burlalt{arXiv:2103.11039
}{https://arxiv.org/abs/2103.11039}.

\bibitem{TZ}
{\rm N.~Tapia, L.~Zambotti}.
\newblock \emph{The geometry of the space of branched rough paths}.
\newblock P. Lond. Math. Soc. \textbf{121}, no.~2, (2020), 220--251.
\newblock
  \burlalt{doi:10.1112/plms.12311}{https://dx.doi.org/10.1112/plms.12311}.
  
  \bibitem{WZ69}
W. ~Zimmermann.
\newblock{\em Convergence of Bogoliubov’s method of renormalization in momentum space.}
\newblock {Comm. Math. Phys. \textbf{15}, (1969), 208--234.}
\newblock \burlalt{doi:10.1007/BF01645676}{http://dx.doi.org/10.1007/BF01645676}.
 
 

\end{thebibliography}
\end{document}